\documentclass{groups17}
\makeatletter
\newcommand{\imod}[1]{\allowbreak\mkern4mu({\operator@font mod}\,\,#1)}
\makeatother

\usepackage{tikz}
\usepackage{longtable}

\renewcommand{\a}{\alpha}
\renewcommand{\b}{\beta}
\renewcommand{\L}{\Lambda}
\renewcommand{\l}{\lambda}
\newcommand{\leqs}{\leqslant}
\newcommand{\geqs}{\geqslant}
\newcommand{\vs}{\vspace{3mm}}
\newcommand{\la}{\langle}
\newcommand{\ra}{\rangle}
\newcommand{\R}{\mathbb{R}}
\newcommand{\C}{\mathcal{C}}
\newcommand{\Z}{\mathbb{Z}}

\begin{document}

\maketitle{IRREDUCIBLE SUBGROUPS OF SIMPLE ALGEBRAIC \\ GROUPS -- A
  SURVEY}{%
  TIMOTHY C. BURNESS$^{\ast}$ and DONNA M. TESTERMAN$^{\dagger}$}{%
  $^{\ast}$School of Mathematics, University of Bristol, Bristol BS8
  1TW, UK \newline
  Email: t.burness@bristol.ac.uk \\[2pt]
  $^{\dagger}$Institut de Math\'ematiques, Station 8, \'{E}cole
  Polytechnique F\'{e}d\'{e}rale de Lausanne, \\ CH-1015 Lausanne,
  Switzerland
  \newline
  Email: donna.testerman@epfl.ch}

\begin{abstract}
Let $G$ be a simple linear algebraic group over an algebraically closed field $K$ of characteristic $p \geqs 0$, let $H$ be a proper closed subgroup of $G$ and let $V$ be a nontrivial finite dimensional irreducible rational $KG$-module. We say that $(G,H,V)$ is an irreducible triple if $V$ is irreducible as a $KH$-module. Determining these triples is a fundamental problem in the representation theory of algebraic groups, which arises naturally in the study of the subgroup structure of classical groups. In the 1980s, Seitz and Testerman extended earlier work of Dynkin on connected subgroups in characteristic zero to all algebraically closed fields. In this article we will survey recent advances towards a classification of irreducible triples for all positive dimensional subgroups of simple algebraic groups. 
\end{abstract}

\section{Introduction}\label{s:intro}

Let $G$ be a group, let $H$ be a proper subgroup of $G$ and let $V$ be a nontrivial finite dimensional irreducible $KG$-module for some field $K$. Let us write $V|_{H}$ to denote the restriction of $V$ to $H$. Then $(G,H,V)$ is an \emph{irreducible triple} if and only if $V|_{H}$ is irreducible. In this survey article we study irreducible triples in the special case where $G$ is a simple linear algebraic group over an algebraically closed field $K$ of characteristic $p \geqs 0$, $H$ is positive dimensional and closed, and $V$ is a rational $KG$-module. Throughout this article all algebraic groups considered are linear; henceforth, we use the term ``algebraic group'' in place of ``linear algebraic group''.

The study of these triples was initiated by Dynkin \cite{Dynkin0,Dynkin} in the 1950s (for $K = \mathbb{C}$) and it was subsequently extended by Seitz \cite{Seitz2} (classical groups) and Testerman \cite{Test1} (exceptional groups) in the 1980s to arbitrary algebraically closed fields. The latter work of Seitz and Testerman provides a classification of the irreducible triples $(G,H,V)$ in the case where $H$ is connected, with the extra condition that if $G$ is classical then $V$ is not the natural module, nor its dual (this assumption is unavoidable). More recently, several authors have focussed on extending this classification to all positive dimensional subgroups and this goal has now essentially been achieved in the series of papers \cite{BGMT,BGT,BMT,Ford1, Ford2, g_paper}. In this article, we will describe some of these recent advances, highlighting the main ideas and techniques. 

As we will discuss in Section \ref{ss:new}, it has very recently come to light that there is an omission in the statement of the main theorem of \cite{Seitz2}. More specifically, there is a 
gap in the proof of \cite[(8.7)]{Seitz2} and a new infinite family of irreducible triples has been discovered by Cavallin and Testerman \cite{CT} (together with an additional missing case for $(G,H,p) = (C_4, B_3,2)$). With a view towards future applications, one of the goals of this article is to record an amended version of Seitz's theorem from \cite{Seitz2}. In addition, we will show that the new family of triples does not have any effect on the main results in \cite{BGMT, BGT} on irreducible triples for disconnected maximal subgroups (see Section \ref{s:disc} for the details). 

\vs

Irreducible triples arise naturally in the study of the subgroup structure of classical algebraic groups, which has been an area of intensive research for several decades. To illustrate the connection, let us consider the group $G = {\rm SL}(V)$, where $V$ is a finite dimensional vector space over an algebraically closed field $K$, and the problem of classifying the maximal closed positive dimensional subgroups of $G$. Let $H$ be a closed positive dimensional subgroup of $G$. As a special case of an important reduction theorem of Liebeck and Seitz \cite[Theorem 1]{LS_classical} (see Theorem \ref{t:ls}), one of the following holds:
\begin{itemize}
\item[(i)] $H$ is contained in a subgroup $M$, where $M$ is the stabilizer in $G$ of a proper non-zero subspace of $V$, or the stabilizer of a nontrivial direct sum or tensor product decomposition of $V$, or the stabilizer of a nondegenerate form on $V$. 
\item[(ii)] The connected component $H^\circ$ is simple and acts irreducibly and tensor indecomposably on $V$, and does not preserve a nondegenerate form on $V$. 
\end{itemize}

It is straightforward to determine which of the subgroups $M$ as in (i) are maximal among closed subgroups; indeed with very few exceptions, these subgroups are maximal in $G$. However, it is not so easy to determine when subgroups $H$ satisfying (ii) are maximal. If such a subgroup $H$ is not maximal, then it must be properly contained in a closed subgroup $L<G$. Moreover, $L^\circ$ must itself be a simple algebraic group acting irreducibly (and tensor indecomposably) on $V$. In addition, if $L^\circ$ is a classical group then $V$ is not the natural module (nor its dual) for $L^\circ$. Therefore, $(L^\circ, H^\circ,V)$ is an irreducible triple of the form studied by Seitz and Testerman \cite{Seitz2, Test1} and we have now ``reduced'' the determination of the maximal closed positive dimensional subgroups of $G$ to the classification of these irreducible triples. 

\vs

Let us say a few words on the organisation of this article. In Sections \ref{s:rep} and \ref{s:cla} we briefly recall some of the main results and terminology we will need on the representation theory and subgroup structure of simple algebraic groups. In particular, we discuss the parametrization of irreducible modules in terms of highest weights and we present a key theorem of Liebeck and Seitz \cite{LS_classical} on the maximal subgroups of classical algebraic groups (see Theorem \ref{t:ls}). Seitz's theorem \cite{Seitz2} on connected irreducible subgroups is then the main focus of Section \ref{s:con}; we briefly outline some of the main ideas in the proof and we discuss the new examples discovered by Cavallin and Testerman \cite{CT}. This allows us to present a corrected version of Seitz's result (see Theorem \ref{t:seitz}, together with Table $3$ in Section \ref{s:tab}). 

In Section \ref{s:disc} we present recent work of Burness, Ford, Ghandour, Marion and Testerman \cite{BGMT, BGT, Ford1, Ford2}, which has extended Seitz's theorem to disconnected positive dimensional maximal subgroups of classical algebraic groups. Here the analysis naturally falls into two cases according to the type of maximal subgroup (in terms of the methods, there is an important distinction between \emph{geometric} and \emph{non-geometric} maximal subgroups). The main result is Theorem \ref{t:main_disc}, with the relevant irreducible triples presented in Table $4$. We present a couple of concrete cases (see Examples  \ref{ex:c2} and \ref{ex:ng}) to illustrate the main ideas in the proof. Since some of the arguments rely on Seitz's main theorem for connected subgroups, we also take the opportunity to verify the statements of the main theorems in \cite{BGMT,BGT} in light of the new examples discovered in \cite{CT} (see Propositions \ref{p:geomcheck} and \ref{p:nongeomcheck}). Finally, in Section \ref{s:tab} we present the tables referred to in Theorems \ref{t:seitz} and \ref{t:main_disc}.

\paragraph{Acknowledgments.} This survey article is based on the content of the second author's one-hour lecture at the \emph{Groups St Andrews} conference, which was hosted by the University of Birmingham in August 2017. It is a pleasure to thank the organizers of this meeting for their generous hospitality. The second author acknowledges the support of the Swiss National Science Foundation, grant number 200021-156583.
In addition, the material is based upon work supported by the US National Science Foundation under Grant No. DMS-1440140 while the second author was in residence at the Mathematical Sciences Research Institute in Berkeley, California during the Spring 2018 semester. Finally, both authors are very grateful to Mikko Korhonen, Martin Liebeck, Gunter Malle and Gary Seitz for their helpful comments on an earlier draft of this article.

\section{Representation theory}\label{s:rep}

Let $G$ be a simply connected simple algebraic group of rank $n$ defined over an algebraically closed field $K$ of characteristic $p \geqs 0$. In this section we introduce some standard notation and terminology, and we briefly recall the parametrization of finite dimensional irreducible $KG$-modules in terms of highest weights, which will be needed for the statements of the main results in Sections \ref{s:con} and \ref{s:disc}. We refer the reader to \cite[Chapters 15 and 16]{MT} for further details on the basic theory, and \cite{Jantzen} for a more in-depth treatment.

\vs

Let $B=UT$ be a Borel subgroup of $G$ containing a fixed maximal torus $T$ of $G$, where $U$ denotes the unipotent radical of $B$. Let $\Pi(G)=\{\a_1, \ldots, \a_{n}\}$ be a corresponding base of the root system $\Sigma(G)=\Sigma^+(G) \cup \Sigma^{-}(G)$ of $G$, where $\Sigma^+(G)$ and $\Sigma^{-}(G)$ denote the positive and negative roots of $G$, respectively (throughout this article, we adopt the standard labelling of simple roots and Dynkin diagrams given in Bourbaki \cite{Bou}). Moreover, we will often identify a
simple algebraic group $S$ with its root system, writing $S = \Sigma_m$, where $\Sigma_m$ is one of 
\[
A_m, \, B_m, \, C_m, \, D_m, \, E_8, \, E_7, \, E_6, \, F_4, \, G_2,
\]
to mean that $S$ has root system of the given type, and use the terminology ``$S$ is of type $\Sigma_m$''. In particular, for such $S$, if $J$ is a simply connected simple group with the same root system type as $S$, then there is a surjective morphism from $J$ to $S$, 
and the representation theory of $S$ is controlled by that of $J$. 

Let $X(T) \cong \Z^n$ denote the character group of $T$ and let $\{\l_1, \ldots, \l_{n}\}$ be the fundamental dominant weights for $T$ corresponding to our choice of base $\Pi(G)$, so $\la \l_i,\a_j\ra=\delta_{i,j}$ for all $i$ and $j$, where
$$\la \l, \a \ra = 2\frac{(\l,\a)}{(\a,\a)},$$
$(\; ,\,)$ is an inner product on $X(T)_{\R}=X(T) \otimes_{\Z} \R$, and $\delta_{i,j}$ is the familiar Kronecker delta. In addition,  let $s_{\a}:X(T)_{\R} \to X(T)_{\R}$  be the reflection relative to $\a \in \Sigma(G)$, defined by $s_{\a}(\l)=\l-\la \l,\a \ra\a$, and set $s_i=s_{\a_i}$. The finite group
\[
W(G) = \la s_{i} \mid 1 \leqs i \leqs n\ra < \operatorname{GL}(X(T)_{\R})
\]
is called the Weyl group of $G$. In addition, we will write $U_{\a}=\{x_{\a}(c) \mid c \in K\}$ to denote the $T$-root subgroup of $G$ corresponding  to $\a \in \Sigma(G)$. Finally, if $H$ is a closed reductive subgroup of $G$ and $T_{H^\circ}$ is a maximal torus of $H^\circ$ contained in $T$ then we will sometimes write $\l|_{H^\circ}$ to denote the restriction of $\l \in X(T)$ to the subtorus $T_{H^\circ}$. 

Let $V$ be a finite dimensional $KG$-module. The action of $T$ on $V$ can be diagonalized, giving a decomposition
$$V = \bigoplus_{\mu \in X(T)}{V_{\mu}},$$
where
$$V_{\mu}=\{v \in V \mid tv = \mu(t)v \mbox{ for all $t \in T$}\}.$$
A character $\mu \in X(T)$ with $V_{\mu} \neq 0$ is called a weight (or $T$-weight) of $V$, and $V_{\mu}$ is its corresponding weight space. The dimension of $V_{\mu}$, denoted by $m_{V}(\mu)$, is called the multiplicity of $\mu$. We write $\L(V)$ for the set of weights of $V$. For any weight $\mu \in \L(V)$ and any root $\a \in \Sigma(G)$ we have 
\begin{equation}\label{e:mt}
U_{\a}v \subseteq v +\sum_{m \in \mathbb{N}^{+}}V_{\mu+m\a}
\end{equation}
for all $v \in V_{\mu}$. There is a natural action of the Weyl group $W(G)$ on $X(T)$, which in turn induces an action on $\L(V)$. In particular, $\L(V)$ is a union of $W(G)$-orbits, and all weights in a $W(G)$-orbit have the same multiplicity.

By the Lie-Kolchin theorem, the Borel subgroup $B$ stabilizes a $1$-dimensional subspace $\la v^{+} \ra$ of $V$; the action of $B$ on $\la v^{+} \ra$ affords a homomorphism $\chi:B \to K^*$ with kernel $U$. Therefore $\chi$ can be identified with a character $\l \in X(T)$, which is a weight of $V$. If $V$ is an irreducible $KG$-module then the $1$-space $\la v^{+} \ra$ is unique, $V=\la Gv^{+} \ra$, $m_{V}(\l)=1$ and each weight $\mu \in \L(V)$ is obtained from $\l$ by subtracting some positive roots. Consequently, we say that $\l$ is the highest weight of $V$, and $v^+$ is a maximal vector. 

Since $G$ is simply connected, the fundamental dominant weights form a $\Z$-basis for the additive group of all weights for $T$, and a weight $\l$ is said to be dominant if $\l=\sum_{i=1}^{n}a_i\l_i$ and each $a_i$ is a non-negative integer. If $V$ is a finite dimensional irreducible $KG$-module then its highest weight is dominant. Conversely, given any dominant weight $\l$ one can construct a finite dimensional irreducible $KG$-module with highest weight $\l$. Moreover, this correspondence defines a bijection between the set of dominant weights of $G$ and the set of isomorphism classes of finite dimensional irreducible $KG$-modules. For a dominant weight $\l=\sum_{i=1}^{n}a_i\l_i$ we write $L_G(\l)$ for the irreducible $KG$-module with highest weight $\l$. If ${\rm char}(K)=p>0$, we say that $L_G(\l)$ is $p$-restricted (or that $\l$ is a $p$-restricted weight) if $a_i<p$ for all $i$. It will be convenient to say that every dominant weight is $p$-restricted when $p=0$. We will often represent the $KG$-module $L_G(\l)$ by an appropriately labelled Dynkin diagram (see Tables 3 and 4, for example). For instance, 
$$
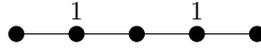
\begin{tikzpicture}[scale=.4, baseline=(current bounding box.base)]
   \tikzstyle{every node}=[font=\small]
   \draw (2,1.4) node[anchor=north] {$1$};
   \draw (6,1.4) node[anchor=north] {$1$};
    \foreach \x in {0,...,4}
    \draw[xshift=\x cm,fill=black] (\x cm,0) circle (.25cm);
    \foreach \y in {0.1,...,3.1}
    \draw[xshift=\y cm] (\y cm,0) -- +(1.6 cm,0);
  \end{tikzpicture}$$
  denotes the irreducible $KG$-module $L_G(\l_2+\l_4)$ for $G = A_5 = {\rm SL}_{6}(K)$.

In characteristic zero, the dimension of $L_G(\l)$ is given by Weyl's dimension formula
\[
\dim L_G(\l) = \prod_{\a \in \Sigma^{+}(G)} \frac{(\a, \l+\rho)}{(\a,\rho)},
\]
where $\rho = \frac{1}{2}\sum_{\a \in \Sigma^{+}(G)}\a$. However, there is no known formula for $\dim L_G(\l)$ when $p>0$ (the above expression is an upper bound), apart from some special cases (see \cite{PS} and \cite[Table 2.1]{BGT}, for example). Using computational methods, L\"{u}beck \cite{Lubeck} has calculated the dimension of every irreducible $KG$-module in all characteristics, up to some fixed dimension depending on the type of $G$. For instance, the irreducible modules for ${\rm SL}_{3}(K)$ of dimension at most $400$ can be read off from \cite[Table A.6]{Lubeck}.

In positive characteristic $p>0$, the Frobenius automorphism of~$K$, \ $F_p:K \to K$, $c \mapsto c^p$, induces an endomorphism $F:G \to G$ of algebraic groups defined by $x_{\a}(c) \mapsto x_{\a}(c^p)$, for all $\a \in \Sigma(G)$, $c \in K$. Given a rational $KG$-module $V$ and corresponding representation $\varphi:G \to {\rm GL}(V)$, we can use $F$ to define a new representation $\varphi^{(p^i)} \colon G \to \operatorname{GL}(V)$ for each integer $i \geqs 1$. 
The corresponding $KG$-module is denoted by $V^{(p^i)}$ and the action is given by $\varphi^{(p^i)}(g)v = \varphi(F^i(g))v$ for $g \in G, v \in V$. We say that $V^{(p^i)}$ is a Frobenius twist of $V$. 

By Steinberg's tensor product theorem (see \cite[Theorem 16.12]{MT}, for example), every finite dimensional irreducible $KG$-module $V$ is a tensor product of Frobenius twists of $p$-restricted $KG$-modules, so it is natural to focus on the $p$-restricted modules. Indeed, if $H$ is a subgroup of $G$ and $V = V_1 \otimes V_2$ is a tensor product of $KG$-modules, then $V$ is irreducible as a $KH$-module only if $V_1$ and $V_2$ are both irreducible $KH$-modules. This observation essentially reduces the problem of determining the irreducible triples for $G$ to the $p$-restricted case, whence the main results in Sections \ref{s:con} and \ref{s:disc} will be stated in terms of $p$-restricted tensor indecomposable $KG$-modules. 

Finally, we record the following technical lemma, which will be relevant later.

\begin{lemma}\label{l:images}
Let $G$ be a simply connected simple algebraic group of type $B_n$ or $C_n$ with $n \geqs 3$, defined over an algebraically closed field of characteristic $2$. Let $\l$ be a $2$-restricted dominant weight for $G$ and let $\varphi : G \to {\rm GL}(V)$ be the associated irreducible representation. Then $\varphi(G)$ is again of type $B_n$, respectively $C_n$, if and only if $\la \l, \a\ra \ne 0$ for some short simple root $\a$ of $G$.
\end{lemma}

\begin{proof}
Let $d\varphi: {\rm Lie}(G) \to \mathfrak{gl}(V)$ be the differential of $\varphi$, where ${\rm Lie}(G)$ is the Lie algebra of $G$. By a theorem of Curtis \cite{Curtis}, $d\varphi$ is irreducible and thus $d\varphi \ne 0$. Recall that ${\rm Lie}(G)$ has a proper ideal containing all short root vectors. This ideal lies in $\ker(d\varphi)$ if and only if $\la \l, \a\ra=0$ for all short simple roots $\a$ of $G$. Moreover, 
${\rm im}(d\varphi)$ is an ideal of ${\rm Lie}(\varphi(G))$ and the result now follows by inspecting the ideal structure of ${\rm Lie}(G)$ (see \cite{Hoge}, for example).
\end{proof}

\section{Subgroup structure}\label{s:cla}

Let $G$ be a simple algebraic group over an algebraically closed field $K$ of characteristic $p \geqs 0$. By Chevalley's classification theorem for semisimple algebraic groups, the isomorphism type of $G$ is determined by its root datum (see \cite[Section 9.2]{MT}) and thus $G$ is either of classical type or exceptional type. In this article, we will focus on irreducible triples for classical type algebraic groups, in which case it will be useful to have a more geometric description of these groups in order to describe their positive dimensional maximal subgroups.

Let $W$ be a finite dimensional vector space over $K$, equipped with a form $\beta$, which is either the zero form, or a symplectic or nondegenerate quadratic form (in the latter case, nondegenerate means that the radical of the underlying bilinear form is trivial, unless $p=2$ and $\dim W$ is odd, in which case the radical is a nonsingular $1$-space). Assume $\dim W \geqs 2$, and further assume that $\dim W \geqs 3$ and $\dim W \ne 4$ if $\b$ is a quadratic form. Let ${\rm Isom}(W,\beta)$ be the corresponding isometry group, comprising the invertible linear maps on $W$ which preserve $\beta$, and let $G = {\rm Isom}(W,\beta)'$ be the derived subgroup. Then $G$ is a simple classical algebraic group over $K$, namely one of ${\rm SL}(W)$, ${\rm Sp}(W)$ or ${\rm SO}(W)$ (we will sometimes use the notation $G = Cl(W)$ to denote any one of these groups). Note that $Cl(W) = {\rm Isom}(W,\beta) \cap {\rm SL}(W)$, unless $p=2$, in which case ${\rm SO}(W)$ has index $2$ in 
${\rm GO}(W) \cap {\rm SL}(W)$. As discussed in the previous section, we will often adopt the Lie notation for groups with the same root system type, so we will write $A_n$, $B_n$, $C_n$, and $D_n$ to denote a simple algebraic group of classical type. 

The main theorem on the subgroup structure of simple classical algebraic groups is due to Liebeck and Seitz \cite{LS_classical}, which can be viewed as an algebraic group analogue of Aschbacher's celebrated subgroup structure theorem \cite{Asch} for finite classical groups. In order to state this result, we need to introduce some additional notation. Following \cite[Section 1]{LS_classical}, one defines six natural, or \emph{geometric}, collections of closed subgroups of $G$, labelled $\C_1, \ldots, \C_6$, and we set $\C=\bigcup_{i}\C_i$. A rough description of the subgroups in each $\C_i$ collection is given in Table \ref{t:subs} (note that the subgroups comprising the collection $\C_5$ are finite). The following result is \cite[Theorem 1]{LS_classical}. 

\begin{table}
\caption{The $\C_i$ collections}\label{t:subs}

\vspace{4mm}

\begin{tabular}{cl} \hline
 & Rough description \\ \hline
$\C_1$ & Stabilizers of subspaces of $W$ \\
$\C_2$ & Stabilizers of orthogonal decompositions $W=\bigoplus_{i}W_i$, $\dim W_i=a$ \\
$\C_3$ & Stabilizers of totally singular decompositions $W=W_1 \oplus W_2$ \\
$\C_4$ & Stabilizers of tensor product decompositions $W=W_1 \otimes W_2$ \\
& Stabilizers of tensor product decompositions $W=\bigotimes_i W_i$, $\dim W_i=a$ \\
$\C_5$ & Normalizers of symplectic-type $r$-groups, $r \neq p $ prime  \\
$\C_6$ & Classical subgroups \\ \hline
\end{tabular}
\end{table}

\begin{thm}\label{t:ls}
Let $H$ be a closed subgroup of $G = Cl(W)$. Then either
\begin{itemize}
\item[{\rm (i)}] $H$ is contained in a member of $\C$; or
\item[{\rm (ii)}] modulo scalars, $H$ is almost simple and $E(H)$ (the unique quasisimple normal subgroup of $H$) is irreducible on $W$. Further, if $G={\rm SL}(W)$ then $E(H)$ does not fix a nondegenerate form on $W$. In addition, if $H$ is infinite then $E(H)=H^{\circ}$ and $W$ is a 
tensor indecomposable $KH^{\circ}$-module.
\end{itemize}
\end{thm}

We refer the reader to \cite{LS_classical} and \cite[Section 2.5]{BGT} for further details on the structure of the geometric subgroups comprising $\mathcal{C}$. We will write $\mathcal{S}$ to denote the collection of closed subgroups of $G$ that arise in part (ii) of Theorem \ref{t:ls} (we sometimes use the term \emph{non-geometric} to refer to the subgroups in $\mathcal{S}$). Note that it is not feasible to determine the members of $\mathcal{S}$, in general. Indeed, as we remarked in the previous section, we do not even know the dimensions of the irreducible modules for a given simple algebraic group. 

\begin{ex}\label{ex:symp}
If $G = {\rm Sp}(W) = {\rm Sp}_{8}(K)$ is a symplectic group, then the positive dimensional subgroups $H \in \mathcal{C} \cup \mathcal{S}$ are as follows:

\begin{itemize}
\item[$\C_1$:]  Here $H$ is either a maximal parabolic subgroup $P_1, \ldots, P_4$, where $P_i$ denotes the stabilizer in $G$ of an $i$-dimensional totally singular subspace of $W$, or $H = {\rm Sp}_{6}(K) \times {\rm Sp}_{2}(K)$ is the stabilizer of a nondegenerate $2$-space. These subgroups are connected.

\item[$\C_2$:] The collection $\mathcal{C}_2$ comprises both the stabilizers of orthogonal decompositions $W = W_1 \perp W_2$, where each $W_i$ is a nondegenerate $4$-space, in which case $H={\rm Sp}_{4}(K) \wr S_2$, together with subgroups of the form $H = {\rm Sp}_{2}(K) \wr S_4$, which stabilize an orthogonal decomposition of $W$ into nondegenerate $2$-spaces.

\item[$\C_3$:]  Here $H = {\rm GL}_{4}(K).2$ is the stabilizer of a direct sum decomposition $W = W_1 \oplus W_2$, where $W_1$ and $W_2$ are totally singular $4$-spaces. 

\item[$\C_4$:]  This collection is empty if $p=2$ (see \cite[p.429]{LS_classical}), so let us assume $p \ne 2$. Here the natural module $W$ admits a tensor product decomposition $W = W_1 \otimes W_2$, where $W_1$ is a $2$-dimensional symplectic space and $W_2$ is a $4$-dimensional orthogonal space. The stabilizer of this decomposition is a (disconnected) central product ${\rm Sp}(W_1) \otimes {\rm GO}(W_2)$. In addition, the remaining subgroups in $\C_4$ are of the form 
\[
({\rm Sp}_{2}(K)\otimes {\rm Sp}_{2}(K) \otimes {\rm Sp}_{2}(K)).S_3,
\]
stabilizing a decomposition $W = W_1 \otimes W_2 \otimes W_3$ into a tensor product of symplectic $2$-spaces.  

\item[$\C_6$:]  If $p=2$ then $H = {\rm GO}(W) = {\rm SO}(W).2$ is the stabilizer of a nondegenerate quadratic form on $W$ (the collection is empty if $p \ne 2$). Note that we view $H$ as a geometric subgroup, even though it satisfies the criteria for membership in $\mathcal{S}$ (see Remark \ref{r:sdef}).

\item[$\mathcal{S}$:]  Here $p \ne 2,3,5,7$ and $H={\rm SL}(U) = {\rm SL}_{2}(K)$ is irreducibly embedded in $G$ via the $8$-dimensional irreducible $KH$-module ${\rm Sym}^7(U)$, the seventh symmetric power of $U$.  
\end{itemize}
\end{ex}

To conclude this section, let us provide a brief sketch of the proof of Theorem \ref{t:ls} in the special case where $G = {\rm SL}(W)$ and $H$ is positive dimensional.

\begin{trivlist}\item[]
\textbf{Sketch proof of Theorem \ref{t:ls}} \
Let us assume $G = {\rm SL}(W)$ and $H$ is positive dimensional. First observe that if $H$ fixes a nondegenerate form on $W$, then $H$ 
lies in a classical subgroup of $G$ belonging to the collection $\C_6$. Therefore, we may assume otherwise. If $H$ acts reducibly on $W$, then $H$ lies in the stabilizer of a proper non-zero subspace of $W$, which is a parabolic subgroup in the collection $\C_1$.

Now assume $H$ acts irreducibly on $W$ and consider the action of the connected component $H^\circ$ on $W$. Since $H$ stabilizes the socle of $W|_{H^\circ}$, it follows that $W|_{H^\circ}$ is completely reducible. Moreover, the irreducibility of $W|_H$ implies that $H$ transitively permutes the homogeneous components of $W|_{H^\circ}$ and thus $H$ stabilizes a direct sum decomposition of the form 
$W = W_1\oplus\cdots\oplus W_t$,
where each $W_i$ is $m$-dimensional for some $m$. In particular, 
$H \leqs ({\rm GL}_m(K) \wr S_t) \cap G$. If $W|_{H^\circ}$ has at least two  homogeneous components (that is, if $t \geqs 2$), then $H$ is contained in a $\C_2$-subgroup of $G$. 

To complete the argument, we may assume $W|_{H^\circ}$ is homogeneous. If $W|_{H^\circ}$ is reducible then $W|_{H^\circ} = U \oplus\cdots\oplus U$ for some irreducible $KH^\circ$-module $U$. Here $\dim U \geqs 2$ (since $H$ is positive dimensional and $Z(G)$ is finite) and further argument shows that $W$ admits a tensor product decomposition $W=W_1\otimes W_2$ such that  $H^{\circ} \leqs  {\rm GL}(W_1) \otimes 1$ and 
$H \leqs N_G({\rm SL}(W_1)\otimes {\rm SL}(W_2))$. In this situation, we conclude that $H$ is contained in a member of the collection $\C_4$. 

Finally, let us assume $H^\circ$ acts irreducibly on $W$. 
Then $H^\circ$ is necessarily reductive as otherwise $H^\circ$ stabilizes the nonzero set of fixed points of its unipotent radical. We then deduce that $H^\circ$ is in fact
semisimple as $Z(H^\circ)$ must act as scalars on $W$ and is therefore finite. So we now have $H^\circ=H_1\cdots H_r$, a commuting
product of simple algebraic groups, and $W|_{H^\circ} =
W_1\otimes\cdots\otimes W_r$, where $W_i$ is an irreducible
$KH_i$-module for each $i$. If $r>1$ then $H$ is contained in the
stabilizer of the tensor product decomposition
$W_1\otimes\cdots\otimes W_r$, which is in the $\C_4$
collection. Finally, suppose $r=1$, so $H^\circ$ is simple. If
$W|_{H^\circ}$ is tensor decomposable then once again we deduce that
$H$ is contained in a $\C_4$-subgroup. Otherwise, $H$ satisfies the
conditions in part (ii) of the theorem and thus $H$ is a member of the
collection $\mathcal{S}$. \hspace *{\fill}$\Box$
\end{trivlist}

\section{Connected subgroups}\label{s:con}

Let $G = Cl(W)$ be a simple classical algebraic group over an arbitrary algebraically closed field $K$ of characteristic $p \geqs 0$. In this section we are interested in the irreducible triples of the form $(G,H,V)$, where $H$ is a closed connected subgroup of $G$ and $V$ is a nontrivial irreducible $p$-restricted rational $KG$-module. The main theorem is due to Seitz \cite{Seitz2}, extending earlier work of Dynkin \cite{Dynkin0, Dynkin} in characteristic zero, which we recall below. We also report on recent work of Cavallin and Testerman \cite{CT}, which has identified a gap in the proof of Seitz's theorem, leading to a new family of irreducible triples. We will explain how the gap arises in \cite{Seitz2} and with a view towards future applications, we will present a corrected version of Seitz's theorem, based on the work in \cite{CT}. 

\subsection{Seitz's theorem}

As described in Section \ref{s:intro}, the determination of the maximal closed positive dimensional
subgroups of the classical algebraic groups relies on a detailed study of irreducible triples. This problem was first investigated in the pioneering work of Dynkin in the 1950s, in the context of complex semisimple Lie algebras. In \cite{Dynkin0,Dynkin}, Dynkin classifies the triples 
$({\mathfrak g},{\mathfrak h},V)$, where ${\mathfrak g}$ is a simple complex  Lie algebra, ${\mathfrak h}$ is a semisimple subalgebra of $\mathfrak g$ and $V$ is an irreducible $\mathfrak g$-module on which $\mathfrak h$ acts irreducibly. For a classical type Lie algebra $\mathfrak g$, his classification includes $8$ infinite families of natural embeddings ${\mathfrak h}\subset{\mathfrak g}$, such as ${\mathfrak {sp}}_{2m}\subset{\mathfrak {sl}_{2m}}$, and in each case he lists the $\mathfrak g$-modules on which $\mathfrak h$ acts irreducibly (in terms of highest weights). In most cases, the module is some symmetric or exterior power of the natural module for $\mathfrak g$.  In addition, there are $36$ isolated examples $({\mathfrak g}, {\mathfrak h}, V)$ with ${\mathfrak g}$ a classical Lie algebra of fixed rank. For example, by embedding the simple Lie algebra $\mathfrak h$ of type $E_6$ in ${\mathfrak {sl}}_{27}={\mathfrak {sl}}(W)$ via one of its irreducible $27$-dimensional representations, he shows that $\mathfrak h$ acts irreducibly on $\Lambda^i(W)$ for $i=2,3,4$. There are exactly $4$ pairs $({\mathfrak g}, {\mathfrak h})$ in Dynkin's theorem for $\mathfrak g$ of exceptional type. Dynkin's classification immediately yields a classification of irreducible triples $(G,H,V)$, where $G$ is a simple algebraic group defined over $\mathbb C$ and $H$ is a proper closed connected subgroup.

In the 1980s, Seitz \cite{Seitz2} (classical groups) and Testerman \cite{Test1} (exceptional groups) extended Dynkin's analysis to all simple algebraic groups~$G$ over all algebraically closed fields. Here, in the positive characteristic setting, several major difficulties arise that are not present in the work of Dynkin, so a completely different approach is needed. For instance, rational modules need not be completely reducible and we have already mentioned that there is no dimension formula for irreducible modules. Similarly, in positive characteristic, irreducible modules may be tensor decomposable (this is clear from Steinberg's tensor product theorem) and this significantly complicates the analysis. 

In order to proceed, it is first useful to observe that if $(G,H,V)$ is an irreducible triple with $H$ connected then $H$ is semisimple. 
Seitz's main technique in \cite{Seitz2} is induction. Given a parabolic subgroup $P_H$ of $H$, the Borel-Tits theorem (see \cite[Corollary 3.9]{BT}) implies the existence of a parabolic subgroup $P$ of $G$ such that $P_H<P$ and $Q_H=R_u(P_H)<R_u(P)=Q$. Let $L_H$ and $L$ be Levi factors of $P_H$ and $P$, respectively. By \cite[(2.1)]{Seitz2}, if $V|_{H}$ is irreducible then $L_H'$ acts irreducibly on the quotient $V/[V,Q]$, which is an irreducible $KL'$-module (here $[V,Q] = \la qv-v \mid v \in V, q \in Q \ra$). Moreover, the highest weight of $V/[V,Q]$ as a $KL'$-module is the restriction of the highest weight of $V$ to an appropriate maximal torus of $L'$ (this is a variation of a result of Smith \cite{Smith}). Thus, Seitz proceeds by induction on the rank of $H$, treating the base case $H=A_1$ by ad hoc methods, exploiting the fact that all weights of an irreducible $KA_1$-module have multiplicity one.

In the main theorem of \cite{Seitz2} (which includes Testerman's results from \cite{Test1} for exceptional groups), the embedding of $H$ in $G$ is described in terms of the action of $H$ on the natural $KG$-module $W$, in case $G$ is of classical type. For an exceptional group $G$,  an explicit construction of the embedding is required, as given in \cite{Test1,Test2}.
Further, the highest weights of $V|_{H}$ and $V = L_G(\l)$ are recorded in terms of fundamental dominant weights for $H$ and $G$ (via labelled Dynkin diagrams). 

Comparing the main theorems of \cite{Dynkin0,Dynkin} and \cite{Seitz2}, we observe that 
Dynkin's classification of triples (for $K = \mathbb{C}$) is a subset of the list of irreducible triples that arise in positive characteristic $p>0$, with some additional conditions depending on $p$. For example, consider the natural embedding of $H={\rm Sp}(W)$ in $G={\rm SL}(W)$. In  characteristic $0$, $H$ acts irreducibly on the symmetric power ${\rm Sym}^i(W)$ for all $i\geqs 0$, whereas in positive characteristic we need the condition 
$0 \leqs i < p$. In addition, there are infinite families of triples that only arise in positive characteristic. For instance, if we take $H<G$ as above and assume $p>2$, then $H$ also acts irreducibly on the $KG$-modules with highest weight $\lambda = a\lambda_k+b\lambda_{k+1}$, where 
\[
\mbox{$1\leqs k<\frac{1}{2}\dim W$ and $a+b=p-1$, with $a\ne 0$ if 
$k=\frac{1}{2}\dim W - 1$}
\]
(see the case labelled ${\rm I}_{1}'$ in \cite[Table 1]{Seitz2}). In positive characteristic we also find that there are two entirely new families of irreducible triples with $(G,H) = (E_6,F_4)$; see the cases labelled ${\rm T}_1$ and ${\rm T}_2$ in \cite[Table 1]{Seitz2}.

\subsection{New irreducible triples}\label{ss:new}

It has very recently come to light that there is a gap in the proof of Seitz's main theorem \cite{Seitz2}, which leads to a new family of irreducible triples. As explained by Cavallin and Testerman in \cite{CT}, the mistake arises in the proof of \cite[(8.7)]{Seitz2}, which concerns the special case where $G = D_{n+1}$ and $H = B_{n}$ with $n \geqs 2$ and $H$ is embedded in $G$ in the usual way (as the stabilizer of a nonsingular $1$-space); this is the case labelled ${\rm IV}_1'$ in \cite[Table 1]{Seitz2}. In the proof of \cite[(8.7)]{Seitz2}, Seitz defines a certain weight vector in $V = L_G(\l)$ of weight different from~$\lambda$ and shows that this vector is annihilated by all simple root vectors in the Lie algebra of $H$, from which he deduces that $V|_{H}$ is reducible. However, if the coefficients in $\l$ satisfy certain congruence conditions, then this vector is in fact zero and so does not give rise to a second composition factor of $V|_{H}$ as claimed. 

The proof of \cite[(8.7)]{Seitz2} is corrected in \cite{CT}, where all the irreducible triples with $(G,H) = (D_{n+1}, B_{n})$ are determined. This is where the new family of examples arises. In Table \ref{tab:41}, the first row corresponds to the original version of Case ${\rm IV}_1'$ in \cite[Table 1]{Seitz2}, which is visibly a special case of the corrected version given in the second row when $n \geqs 3$.

\begin{table}
\caption{Case ${\rm IV}_1'$ in \cite[Table 1]{Seitz2} (see (a)) and the corrected version (b)}
\label{tab:41}

\vspace{-3mm}

{\small
$$\begin{array}{llll} \hline
& V|_{H} & V & \mbox{Conditions} \\ \hline
{\rm (a)}  & 
\begin{tikzpicture}[scale=.35, baseline=(current bounding box.base)]
   \tikzstyle{every node}=[font=\small]
   \draw (2,1.4) node[anchor=north] {$a_i$};
   \draw (6,1.4) node[anchor=north] {$a_n$};
    \foreach \x in {0,...,3}
    \draw[xshift=\x cm,fill=black] (\x cm,0) circle (.25cm);
    \draw[xshift=3 cm,fill=black] (3 cm, 0) circle (.25cm);
    \draw[dotted,thick] (0.55cm,0) -- +(1cm,0);
    \draw[dotted,thick] (2.55,0) -- +(1cm,0);
    \draw (4.2 cm, .1 cm) -- +(1.6 cm,0);
    \draw (4.2 cm, -.1 cm) -- +(1.6 cm,0);
    \draw (5.3,0) --+ (-0.7,0.4);
    \draw (5.3,0) --+ (-0.7,-0.4);
  \end{tikzpicture}
   & 
   \begin{tikzpicture}[scale=.35, baseline=(current bounding box.base)]
   \tikzstyle{every node}=[font=\small]
   \draw (2,1.4) node[anchor=north] {$a_i$};
   \draw (7.3,1.2) node[anchor=east] {$a_n$};
     \foreach \x in {0,...,2}
    \draw[xshift=\x cm,fill=black] (\x cm,0) circle (.25cm);
    \draw[xshift=4 cm,fill=black] (30: 17 mm) circle (.25cm);
    \draw[xshift=4 cm,fill=black] (-30: 17 mm) circle (.25cm);
     \draw[dotted,thick] (0.55cm,0) -- +(1cm,0);
    \draw[dotted,thick] (2.55cm,0) -- +(1cm,0);
    \draw[xshift=4 cm] (30: 2mm) -- (30: 16mm);
    \draw[xshift=4 cm] (-30: 2mm) -- (-30: 16mm); 
       \end{tikzpicture} 
       & \hspace{-1.5mm} \begin{array}{l} \mbox{$a_ia_n \ne 0$, $1 \leqs i < n$ and} \\ a_i+a_n+n-i \equiv 0 \ \text{(mod~$p$)} \end{array} \\
              & & & \\
{\rm (b)}  & \hspace{-1.8mm}
   \begin{tikzpicture}[scale=.35, baseline=(current bounding box.base)]
  \tikzstyle{every node}=[font=\small]
  \draw (0,1.4) node[anchor=north] {$a_1$};
   \draw (2,1.4) node[anchor=north] {$a_2$};
   \draw (5.8,1.4) node[anchor=north] {$a_{n-1}$};
   \draw (8,1.4) node[anchor=north] {$a_n$};
    \foreach \x in {0,...,1}
    \draw[xshift=\x cm,fill=black] (\x cm,0) circle (.25cm);
    \draw[xshift=3 cm,fill=black] (3 cm, 0) circle (.25cm);
    \draw[xshift=5 cm,fill=black] (3 cm, 0) circle (.25cm);
    \draw[dotted,thick] (2.55cm,0) -- +(3cm,0);
    \foreach \y in {0.1,...,0.1}
    \draw[xshift=\y cm] (\y cm,0) -- +(1.6cm,0);
    \draw (6.2 cm, .1 cm) -- +(1.6 cm,0);
    \draw (6.2 cm, -.1 cm) -- +(1.6 cm,0);
    \draw (7.3,0) --+ (-0.7,0.4);
    \draw (7.3,0) --+ (-0.7,-0.4);
         \end{tikzpicture}   & \hspace{-1.8mm}
   \begin{tikzpicture}[scale=.35, baseline=(current bounding box.base)]
   \tikzstyle{every node}=[font=\small]
   \draw (0,1.4) node[anchor=north] {$a_1$};
   \draw (2,1.4) node[anchor=north] {$a_2$};
   \draw (5.4,1.4) node[anchor=north] {$a_{n-1}$};
   \draw (9.4,1.2) node[anchor=east] {$a_{n}$};
   \draw (8.9,-1.2) node[anchor=east] {};
    \foreach \x in {0,...,1}
    \draw[xshift=\x cm,fill=black] (\x cm,0) circle (.25cm);
    \draw[xshift=2 cm,fill=black] (4,0) circle (.25cm);
    \draw[xshift=6 cm,fill=black] (30: 17 mm) circle (.25cm);
    \draw[xshift=6 cm,fill=black] (-30: 17 mm) circle (.25cm);
    \draw[dotted,thick] (2.55cm,0) -- +(3cm,0);
    \foreach \y in {0.1,...,0.1}
    \draw[xshift=\y cm] (\y cm,0) -- +(1.6cm,0);
    \draw[xshift=6 cm] (30: 2mm) -- (30: 16mm);
    \draw[xshift=6 cm] (-30: 2mm) -- (-30: 16mm);   
    \end{tikzpicture}
& \hspace{-1.5mm} \begin{array}{l}  
\mbox{$a_n \ne 0$, and} \\
\mbox{if $1 \leqs i < j \leqs n$, $a_ia_j \ne 0$ and} \\
\mbox{$a_k=0$ for all $i<k<j$, then}\\
\mbox{$a_i+a_j+j-i \equiv 0 \ \text{(mod~$p$)}$} 
\end{array} \\ \hline
       \end{array}$$}
       \end{table}

Since the proofs of the main theorems of \cite{Seitz2,Test1} are inductive, the existence of a new family of examples changes the inductive hypothesis, leading to the possible existence of additional irreducible triples. It is shown in \cite{CT} that the new examples with $(G,H) = (D_{n+1},B_{n})$ have no further influence on the statements of the main theorems in \cite{Seitz2, Test1}, under the assumption that the remainder of Seitz's proof is valid (see Hypothesis 1.4 of \cite{CT} for further details).

In \cite{CT}, Cavallin and Testerman identify another missing case $(G,H,V)$ in the statement of Seitz's main theorem. Here $G = C_4$, $H = B_3$, $p=2$ and $V$ is the $48$-dimensional irreducible $KG$-module with highest weight $\l_3$:
$$\begin{array}{lll}
V|_{H} = \begin{tikzpicture}[scale=.4, baseline=-0.8mm]
   \tikzstyle{every node}=[font=\small]
   \draw (0,1.4) node[anchor=north] {$1$};
   \draw (4,1.4) node[anchor=north] {$1$};
    \foreach \x in {0,...,2}
    \draw[xshift=\x cm,fill=black] (\x cm,0) circle (.25cm);
    \draw[xshift=2 cm,fill=black] (2 cm, 0) circle (.25cm);
    \foreach \y in {0.1,...,0.1}
    \draw[xshift=\y cm] (\y cm,0) -- +(1.6cm,0);
    \draw (2.2 cm, .1 cm) -- +(1.6 cm,0);
    \draw (2.2 cm, -.1 cm) -- +(1.6 cm,0);
    \draw (3.3,0) --+ (-0.7,0.4);
    \draw (3.3,0) --+ (-0.7,-0.4);
           \end{tikzpicture}
& \hspace{2mm}  &
V = \begin{tikzpicture}[scale=.4, baseline=-0.95mm]
   \tikzstyle{every node}=[font=\small]
   \draw (4,1.4) node[anchor=north] {$1$};
    \foreach \x in {0,...,2}
    \draw[xshift=\x cm,fill=black] (\x cm,0) circle (.25cm);
    \draw[xshift=3 cm,fill=black] (3 cm, 0) circle (.25cm);
    \foreach \y in {0.1,...,1.1}
    \draw[xshift=\y cm] (\y cm,0) -- +(1.6cm,0);
    \draw (4.2 cm, .1 cm) -- +(1.6 cm,0);
    \draw (4.2 cm, -.1 cm) -- +(1.6 cm,0);
    \draw (4.8,0) --+ (0.72,0.4);
    \draw (4.8,0) --+ (0.72,-0.4);
             \end{tikzpicture}  
\end{array}$$

As explained in \cite{CT}, this is an isolated example. 
Here $H=B_3$ is contained in a maximal rank subgroup of type $D_4<C_4$, both acting irreducibly on $V$. The triples 
$(C_4,D_4,V)$ and $(D_4,B_3,V|_{D_4})$ both appear in \cite[Table 1]{Seitz2}, but 
$(C_4,B_3,V)$ has been omitted. 
This oversight occurs in  the proof of \cite[(15.13)]{Seitz2}. In the first part of the argument,  there is a reduction to the case where 
$H$ acts irreducibly with highest weight~$\delta_{3}$ (the third fundamental dominant weight for~$H$) on the natural $KG$-module, so $G=C_4$ or $D_4$. Now arguing with an appropriate parabolic embedding and using induction as usual, Seitz deduces that 
$V = L_G(\lambda_3)$. But here he concludes that $V$ is a spin module for $G$, and so he is only considering the embedding of $H$ in $D_4$, omitting to include the case of $H$ in $C_4$.

\subsection{A revised version of Seitz's theorem}

In view of \cite{CT}, together with the discussion in Section \ref{ss:new}, we are now in a position to state a corrected (and slightly modified) version of \cite[Theorem 1]{Seitz2}. Note that Table $3$ is located in Section \ref{s:tab}.

\begin{thm}\label{t:seitz}
Let $G$ be a simply connected simple algebraic group over an algebraically closed field $K$ of characteristic $p \geqs 0$. Let $\varphi: G \to {\rm SL}(V)$ be an irreducible, rational, tensor indecomposable $p$-restricted representation and assume that $V \ne W, W^*$ if $G = Cl(W)$ is a classical group. Let $H$ be a closed connected proper subgroup of $G$. Then $V|_{H}$ is irreducible if and only if $(\varphi(G),\varphi(H),V)$ is one of the cases recorded in Table $3$. 
\end{thm}

\begin{re}\label{r:seitz}
Let us make some detailed comments on the statement of this theorem. We refer the reader to Remark \ref{r:tab5} in Section \ref{s:tab} for further guidance on how to read Table $3$ and the notation therein.

\begin{itemize}
\item[(a)] First, let us consider the hypotheses. As explained in Section \ref{s:rep}, it is natural to assume that the irreducible $KG$-module $V$ is both $p$-restricted and tensor indecomposable. The additional condition $V \ne W, W^*$ when $G = Cl(W)$ is a classical group is unavoidable. For example, it is not feasible to determine the simple subgroups of $G$ that act irreducibly on the natural module. 

\vspace{2mm}

There is one particular case worth highlighting here,
namely when $G=B_2=C_2$. Here there are ``two'' natural modules, the $4$-dimensional symplectic 
module and the $5$-dimensional orthogonal module (if $p\ne 2$). There is in fact one example of a 
positive dimensional subgroup $H < G$ acting irreducibly on both of these modules, namely $H=A_1$, when $p \ne 2,3$. This example is not recorded in Seitz's original table, and for the reasons explained above, we do not include it in Table $3$.

\item[(b)] To be consistent with \cite[Theorem 1]{Seitz2}, we have stated Theorem \ref{t:seitz} in terms of a representation $\varphi:G \to {\rm SL}(V)$, and the triples in Table $3$ are given in the form $(\varphi(G),\varphi(H),V)$, rather than $(G,H,V)$. Notice that the only possible difference in recording the triples in this manner arises when $p=2$ and $G$ (or $H$) is of type $B_n$ or $C_n$, in which case the type of $\varphi(G)$ (and $\varphi(H)$) can be determined by applying Lemma \ref{l:images}. 

\vspace{2mm}

For example, consider the case labelled ${\rm IV}_9$ in \cite[Table 1]{Seitz2}, with $p=2$.  Set $H = C_4$ and let $\{\delta_1, \ldots, \delta_4\}$ be a set of fundamental dominant weights for $H$. Consider the $16$-dimensional $KH$-module $W = L_H(\delta_4)$ and let $\rho: H \to {\rm SL}(W)$ be the corresponding representation. Then 
$\rho(H)$ preserves a nondegenerate quadratic form on $W$, so we have 
$\rho(H)<G=D_8$. (This follows, for example, from \cite[2.4]{Sin-Willems} or \cite[10.1.1]{gari-nakano} and the structure of the associated Weyl module.) Let 
$\varphi: G \to {\rm SL}(V)$ be the irreducible representation of highest weight $\lambda_7$. Then $\rho(H)$ acts irreducibly on $V$. By applying Lemma \ref{l:images}, we see that $\rho(H)$ is a simple algebraic group of type $B_4$. Moreover, $V|_{B_4}$ has highest weight $\eta_1+\eta_4$ (where $\eta_i$, $1\leqs i\leqs 4$, are the fundamental dominant weights for $B_4$), and a further application of Lemma \ref{l:images} implies that $(\varphi(G),\varphi(\rho(H)))$ is indeed $(D_8,B_4)$, as listed in Table $3$.

\item[(c)] Clearly, $V|_{H}$ is irreducible if and only if $V^*|_{H}$ is irreducible. Therefore, in order to avoid unnecessary repetitions, the triples in Table $3$ are recorded up to duals. To clarify this set-up, it may be helpful to recall that 
\[
L_G(\l)^* \cong L_G(-w_0(\l))
\] 
for any dominant weight $\lambda$, where $w_0$ is the longest element of the Weyl group $W(G)$. The action of $w_0$ on the root system of $G$ is as follows:
\[
w_0 = \left\{\begin{array}{ll}
-1 & \mbox{for $G$ of type $A_1$, $B_n$, $C_n$, $D_n$ ($n$ even), $G_2$, $F_4$, $E_7$ or $E_8$} \\
-\tau & \mbox{for $G$ of type $A_n$ $(n\geqs 2)$, $D_n$ ($n$ odd, $n \geqs 3$) or $E_6$,}
\end{array}\right.
\]
where $\tau$ is the standard involutory symmetry of the corresponding Dynkin diagram. In particular, note that every $KG$-module $L_G(\l)$ is self-dual when $w_0=-1$. For the cases where $V=L_G(\lambda)$ is self-dual, we must consider separately the action of a fixed subgroup on the module whose highest weight 
is given by $\tau(\lambda)$, when the diagram of $G$ does admit an involutory symmetry. This is usually straightforward and is discussed in the next item.  

\item[(d)] In view of the previous comment, the cases where $G = D_{2m}$ and the highest weight~$\lambda$ of $V$ is not invariant under the symmetry $\tau$ require special consideration. The modules $L_{G}(\lambda)$ and $L_{G}(\tau(\lambda))$ are self-dual and so we need to determine if a given subgroup $H<G$ acts irreducibly on both of them.
In the notation of \cite[Table~1]{Seitz2}, this concerns the following cases:
\begin{equation}\label{e:list}
{\rm IV}_1, \:{\rm IV}_1', \; {\rm IV}_2, \; {\rm IV}_{2}', \; {\rm IV}_3, \; {\rm IV}_9, \; {\rm S}_6.
\end{equation}

\vspace{2mm}

First consider the case labelled ${\rm IV}_9$, where $(G,H) = (D_8,B_4)$ and $p \ne 3$. By considering weight restrictions, we find that $H$ acts irreducibly on exactly one of the two spin modules (and therefore, a non-conjugate copy of $H$, namely the image of $H$ under an involutory graph automorphism of $G$, acts irreducibly on the other spin module). For the other cases in \eqref{e:list}, $H$ lies in the fixed point subgroup of 
an outer automorphism of $G$, and hence acts irreducibly (with the same highest weight) on both $L_G(\l)$ and $L_G(\tau(\l))$. 

\vspace{2mm}

In the latter situation, where $H$ acts irreducibly on both spin modules, we only record $L_G(\l_{2m-1})$ in Table $3$. In order to distinguish the special case ${\rm IV}_9$, we write $(\dagger)$ in the second column to indicate that $H$ only acts irreducibly on one of the spin modules.

\item[(e)] The case $G=D_4$ requires special attention. If $p\ne 3$, the group $H=A_2$ embeds
in $G$, via the irreducible adjoint representation. In \cite[Table 1]{Seitz2}, Seitz does not 
include the action of $H$ on the two $8$-dimensional spin modules for $D_4$, which are the images of the natural module under a triality graph automorphism. However, by considering weight restrictions we see that $H$
acts irreducibly on all three $8$-dimensional irreducible $KG$-modules. This configuration satisfies the hypotheses of Theorem \ref{t:seitz}, so it is included in Table $3$ (see the case labelled ${\rm S}_{11}$). Since $H$ acts irreducibly on both spin modules, we just list the $KG$-module with highest weight $\l_3$, as per the convention explained in (d).

\item[(f)] Notice that we include exceptional type algebraic groups in the statement of the theorem. This is consistent with the statement of \cite[Theorem 1]{Seitz2}, which incorporates Testerman's results on exceptional groups in \cite{Test1}. Note that the three question marks appearing in \cite[Table 1]{Seitz2} can be removed (the existence of these embeddings in $E_6$ and $F_4$ was established by Testerman in \cite{Test2}). 

\item[(g)]  In \cite[Theorem 4.1]{Seitz2}, Seitz determines all irreducible triples $(G,H,V)$ such that $H$ contains a maximal torus of $G$. These occur precisely when $(G,p)$ is one of the following:
\[
(B_n,2), \, (C_n,2), \, (F_4,2), \, (G_2,3)
\]
and the root system of $H$ (which is naturally viewed as a subsystem of the root system of $G$) contains either all long or all short roots of the root system of $G$. In this situation, $H$ acts irreducibly on $L_G(\l)$ if and only if $\l$ has support on the long, respectively short, roots.
It is well known that a subgroup corresponding to a short root system is conjugate, by an exceptional graph morphism, to a subgroup corresponding to a long root subsystem. In \cite[Table 1]{Seitz2}, Seitz records only 
one of the configurations in each case, namely those corresponding to short roots. However, we include both configurations in Table $3$, with the additional cases  
${\rm MR}_1'$ and ${\rm MR}_2'$ covering the situation where $H$ corresponds to a 
subsystem containing the long roots. We also adopt the standard notation $\widetilde{X}$ to denote a subsystem subgroup $X<G$ corresponding to short roots of $G$.

\vspace{2mm}

Since the cases in Table $3$ are recorded in terms of the image groups under the corresponding representations (see part (b) above), there is no need to list any additional cases for the configurations ${\rm MR}_3$ and ${\rm MR}_4$. For example, let us consider ${\rm MR}_3$, where $G = F_4$ and $p=2$. Let $V$ be the irreducible $KG$-module with highest weight $a\lambda_1+b\lambda_2$ and let $\varphi:G \to {\rm SL}(V)$ be the corresponding representation. By \cite[Theorem  4.1]{Seitz2}, the long root subsystem subgroup of type $H=B_4$ acts irreducibly on $V$, but $\varphi(H)$ is a subgroup of type $C_4$ (see Lemma \ref{l:images}), so this configuration is listed as ${\rm MR}_3$.

\item[(h)] Finally, let us note that we have included the new cases from \cite{CT}, as discussed in  Section \ref{ss:new}. In particular, the conditions in case ${\rm IV}_1'$ have been corrected, and we have added ${\rm S}_{10}$ for $(G,H,p) = (C_4,B_3,2)$. This is in addition to the case ${\rm S}_{11}$ mentioned above in part (e).
\end{itemize}
\end{re}

\section{Disconnected subgroups}\label{s:disc}

Let us continue to assume that $G$ is a simply connected simple algebraic group over an algebraically closed field $K$ of characteristic $p \geqs 0$. It is natural to seek an extension of Seitz's main theorem, which we discussed in the previous section, to all positive dimensional closed subgroups. In this section, we will report on recent progress towards a classification of the irreducible triples of this form. 

We start by recalling earlier work of Ford \cite{Ford1, Ford2} and Ghandour \cite{g_paper} on classical and exceptional groups, respectively. However, our main aim is to discuss the recent results in \cite{BGMT} and \cite{BGT}, where the problem for classical algebraic groups and disconnected maximal subgroups is studied (see \cite{BMT} for further results on non-maximal disconnected subgroups). Since Seitz's main theorem in \cite{Seitz2} plays an important role in the proofs of these results, we also take the opportunity to clarify the status of the main results in light of the new examples discovered in \cite{CT} (see Propositions \ref{p:geomcheck} and \ref{p:nongeomcheck}).

\subsection{Earlier work}

Let $H$ be a positive dimensional disconnected subgroup of $G$ with connected component $H^{\circ}$. Note that in order to determine the irreducible triples $(G,H,V)$, we may as well assume that $HZ(G)/Z(G)$ is disconnected, where $Z(G)$ denotes the center of $G$. These triples were first studied by Ford \cite{Ford1, Ford2} in the 1990s under some additional hypotheses. More precisely, he assumes $G = Cl(W)$ is of classical type and his goal is to determine the irreducible triples $(G,H,V)$ where $H^\circ$ is simple and the restriction $V|_{H^\circ}$ has $p$-restricted composition factors (in \cite{Ford1} he handles the cases where $H^{\circ}$ acts reducibly on $W$, and the irreducible case is treated in \cite{Ford2}). These extra assumptions help to simplify the analysis, but they are somewhat restrictive in terms of possible applications. Nevertheless, under these hypotheses Ford discovered the following interesting family of irreducible triples.

\begin{ex}\label{ex:ford}
(Ford \cite{Ford1}.) Suppose $G=B_n$, $p \ne 2$ and $H=D_n.2$ is embedded in $G$ as the stabilizer of a nonsingular $1$-space. As explained in \cite[Section 3]{Ford1}, there is a family of irreducible triples $(G,H,V)$, where $V = L_G(\l)$ with $\l = \sum_{i}a_i\l_i$ and
$$V|_{H^{\circ}} = V_1 \oplus V_2 = \begin{tikzpicture}[scale=.4, baseline=(current bounding box.base)]
   \tikzstyle{every node}=[font=\small]
   \draw (0,1.4) node[anchor=north] {$a_1$};
   \draw (2,1.4) node[anchor=north] {$a_2$};
   \draw (5.4,1.4) node[anchor=north] {$a_{n-2}$};
   \draw (10,1.2) node[anchor=east] {$a_{n-1}$};
   \draw (10,-1.4) node[anchor=east] {$a_{n-1}+1$};
    \foreach \x in {0,...,1}
    \draw[xshift=\x cm,fill=black] (\x cm,0) circle (.25cm);
    \draw[xshift=2 cm,fill=black] (4,0) circle (.25cm);
    \draw[xshift=6 cm,fill=black] (30: 17 mm) circle (.25cm);
    \draw[xshift=6 cm,fill=black] (-30: 17 mm) circle (.25cm);
    \draw[dotted,thick] (2.55cm,0) -- +(3cm,0);
    \foreach \y in {0.1,...,0.1}
    \draw[xshift=\y cm] (\y cm,0) -- +(1.6cm,0);
    \draw[xshift=6 cm] (30: 2mm) -- (30: 16mm);
    \draw[xshift=6 cm] (-30: 2mm) -- (-30: 16mm);   
    \end{tikzpicture}
    \oplus 
    \begin{tikzpicture}[scale=.4, baseline=(current bounding box.base)]
   \tikzstyle{every node}=[font=\small]
   \draw (0,1.4) node[anchor=north] {$a_1$};
   \draw (2,1.4) node[anchor=north] {$a_2$};
   \draw (5.4,1.4) node[anchor=north] {$a_{n-2}$};
   \draw (10,1.5) node[anchor=east] {$a_{n-1}+1$};
   \draw (10,-1.2) node[anchor=east] {$a_{n-1}$};
    \foreach \x in {0,...,1}
    \draw[xshift=\x cm,fill=black] (\x cm,0) circle (.25cm);
    \draw[xshift=2 cm,fill=black] (4,0) circle (.25cm);
    \draw[xshift=6 cm,fill=black] (30: 17 mm) circle (.25cm);
    \draw[xshift=6 cm,fill=black] (-30: 17 mm) circle (.25cm);
    \draw[dotted,thick] (2.55cm,0) -- +(3cm,0);
    \foreach \y in {0.1,...,0.1}
    \draw[xshift=\y cm] (\y cm,0) -- +(1.6cm,0);
    \draw[xshift=6 cm] (30: 2mm) -- (30: 16mm);
    \draw[xshift=6 cm] (-30: 2mm) -- (-30: 16mm);   
    \end{tikzpicture}
    $$
  subject to the following conditions: 
  \begin{itemize}
  \item[(a)] $a_n=1$;
  \item[(b)] if $a_i,a_j \neq 0$, where $i<j<n$ and $a_k=0$ for all $i<k<j$, then $a_i+a_j \equiv i-j \imod{p}$; 
  \item[(c)] if $i<n$ is maximal such that $a_i \neq 0$ then $2a_i \equiv -2(n-i)-1 \imod{p}$.
  \end{itemize}
It is interesting to note that this family of examples has applications in the representation theory of the symmetric groups, playing a role in the proof of the Mullineux conjecture in \cite{FK}. 
\end{ex}

For exceptional groups, Ghandour \cite{g_paper} has extended Testerman's work in \cite{Test1} to all positive dimensional subgroups. She adopts an inductive approach, using Clifford theory and a combinatorial analysis of weights and their restrictions to a suitably chosen maximal torus of the semisimple group $[H^{\circ},H^{\circ}]$ (it is easy to see that the irreducibility of $V|_{H}$ implies that $H^{\circ}$ is reductive; see \cite[Lemma 2.10]{g_paper}, for example). First she handles the cases where $H$ is a disconnected maximal subgroup of $G$ (this uses the classification of the maximal positive dimensional subgroups of simple algebraic groups of exceptional type, which was completed by Liebeck and Seitz in \cite{LS04}). For a non-maximal subgroup $H$, she studies the embedding 
$$H<M<G$$
of $H$ in a maximal subgroup $M$, noting that $(G,H,V)$ is an irreducible triple only if $V|_{M}$ is irreducible. Now the possibilities for $(G,M,V)$ are known, either by the first part of the argument (if $M$ is disconnected) or by the main theorem of \cite{Test1} (when $M$ is connected). There are very few triples $(G,M,V)$ of this form; in each case, $M$ is a commuting product of small rank classical groups and Ghandour now goes on to identify all irreducible triples 
$(M,H,V)$ to conclude.

\subsection{The main result}

In view of this earlier work, the main outstanding challenge is to extend Ford's results on classical groups in \cite{Ford1, Ford2} by removing the conditions on the structure of $H^\circ$ and the composition factors of $V|_{H^\circ}$. As for exceptional groups, our first goal is to determine the irreducible triples $(G,H,V)$ in the case where $H$ is a disconnected maximal  subgroup (and then the general situation can be studied inductively). As mentioned above, in this article we will focus on the crucial first step in this program.

Before giving a brief overview of the results and methods, it might be helpful to make a couple of basic observations. Let $G = Cl(W)$ be a classical type algebraic group. Firstly, if $(G,H,V)$ is an irreducible triple then $H^{\circ}$ is reductive (as noted above). Secondly, it is easy to determine the relevant triples such that $V|_{H^{\circ}}$ is irreducible from Seitz's main theorem in \cite{Seitz2}. 

\begin{ex}
For instance, consider the case labelled ${\rm II}_{1}$ in \cite[Table 1]{Seitz2}, so $G = C_{10}$, $H^{\circ} = A_5$, $V = L_G(\l_2)$ and $p \ne 2$, with $H^{\circ}$ embedded in $G$ via the $20$-dimensional module $W=\L^3(U)$, the third exterior power of the natural module $U$ for $H^{\circ}$. Since the $KH^{\circ}$-module $W$ is self-dual, it follows that the disconnected group $H = H^{\circ}\la t \ra = H^{\circ}.2$ also acts on $W$, where $t$ is a graph automorphism of $H^{\circ}$. This affords an embedding of $H$ in $G$, which yields an irreducible triple $(G,H,V)$ with $H$ disconnected. 
\end{ex}

As in \cite{Seitz2}, it is not feasible to determine the irreducible triples $(G,H,V)$ with 
$V = W$ or $W^*$. Given these observations, it is natural to work with the following hypothesis. 

\paragraph{Hypothesis \boldmath$(\star)$.} 
\emph{$G = Cl(W)$ is a simply connected simple algebraic group of classical type defined over an algebraically closed field $K$ of characteristic $p \geqs 0$, $H$ is a maximal closed positive dimensional subgroup of $G$ such that $HZ(G)/Z(G)$ is disconnected and $V \ne W,W^*$ is a rational tensor indecomposable $p$-restricted irreducible $KG$-module with highest weight $\l$ such that $V|_{H^\circ}$ is reducible.}

\vs

In view of Theorem \ref{t:ls}, the analysis of the irreducible triples $(G,H,V)$ satisfying the conditions in Hypothesis $(\star)$ naturally falls into two cases; either $H$ is geometric or non-geometric. Recall that $H^{\circ}$ is reductive, fix a maximal torus $T$ of $G$ and suppose $V|_{H}$ is irreducible.

\paragraph{Case 1.} \emph{Geometric subgroups.}

\vs

First we assume $H$ is a geometric subgroup of $G$, so in terms of Theorem \ref{t:ls} we have $H \in \C_i$ for some $i \in \{1,2,3,4,6\}$ and in \cite{BGT} we consider each collection in turn. Typically, the explicit description of the embedding of $H$ in $G$ allows us to compute the restrictions of $T$-weights of $V$ to a suitable maximal torus of the semisimple subgroup $[H^{\circ},H^{\circ}]$. By Clifford theory, the highest weights of the $KH^\circ$-composition factors of $V$ are conjugate under the induced action of $H/H^{\circ}$ and this observation allows us to severely restrict the possibilities for $V$ through a combinatorial analysis of weights and their restrictions (see Section \ref{ss:geom} for more details).

\paragraph{Case 2.} \emph{Non-geometric subgroups.}

\vs

To complete the analysis of maximal subgroups, we turn to the non-geometric almost simple irreducible subgroups comprising the collection $\mathcal{S}$ in Theorem \ref{t:ls}. Here the approach outlined in Case 1 is no longer available since we do not have a concrete description of the embedding of $H$ in $G$, in general. This situation is handled in \cite{BGMT}, where we proceed by adapting Seitz's inductive approach in \cite{Seitz2} for connected subgroups, based on a detailed analysis of parabolic subgroups and their embeddings. We will say more about this in Section \ref{ss:ngeom} below.

\begin{re}\label{r:sdef}
It will be convenient to impose an additional condition on the non-geometric subgroups comprising the collection $\mathcal{S}$ (this corresponds to the condition ${\rm S}4$ on p.3 
of \cite{BGMT}). Namely, if $G = {\rm Sp}(W)$, $p=2$ and $H \in \mathcal{S}$, then we will assume that $H^{\circ}$ does not fix a nondegenerate quadratic form on $W$. In particular, we will view the natural embedding ${\rm GO}(W)<{\rm Sp}(W)$ as a geometric maximal subgroup in the collection $\C_6$. 
\end{re}

We can now present the main result on irreducible triples for classical algebraic groups and maximal positive dimensional disconnected subgroups, which combines the main theorems in \cite{BGMT} and \cite{BGT}. Note that Table $4$ is given in Section \ref{s:tab}.

\begin{thm}\label{t:main_disc}
Let $G$, $H$ and $V = L_G(\l)$ be given as in Hypothesis $(\star)$, and let $\varphi:G \to {\rm SL}(V)$ be the representation afforded by $V$. Then $V|_{H}$ is irreducible if and only if $(\varphi(G),\varphi(H),V)$ is one of the cases recorded in Table $4$.
\end{thm}

\begin{re}\label{r:ndisc}
Let us highlight some features of the statement of this theorem.
\begin{itemize}
\item[(a)] First notice that the condition $V \ne W,W^*$ is part of Hypothesis $(\star)$. As before, this is unavoidable if we want to try to determine the cases where the maximal subgroup $H$ is non-geometric (indeed, in this situation, $H^{\circ}$~acts irreducibly on~$W$). However, it is straightforward to read off the disconnected geometric subgroups that act irreducibly on the natural $KG$-module (or its dual); see  \cite[Table 1]{BGT}, for example. 

\item[(b)] In order to be consistent with the statement of Theorem \ref{t:seitz}, we have slightly modified the presentation of the main results in \cite{BGMT, BGT}, in the sense that we give the irreducible triples in the form $(\varphi(G), \varphi(H), V)$. To do this, we use Lemma 
\ref{l:images}. See part (e) below for further comments.

\item[(c)] Almost all of the irreducible triples in Table $4$ arise from geometric maximal subgroups in one of the $\C_i$ collections. In fact, there is only one irreducible triple satisfying the conditions in Hypothesis $(\star)$ with $H \in \mathcal{S}$. This is the case labelled ${\rm N}_{1}$ in Table $4$, where 
\[
\mbox{$G = C_{10}$, $H = A_5.2$ and $\l=\l_3$, with $p \neq 2,3$ and $W|_{H^{\circ}} = L_{H^{\circ}}(\delta_3)$}
\]
with respect to a set of fundamental dominant weights $\{\delta_1, \ldots, \delta_5\}$ for $H^{\circ}$. 

\item[(d)] We use labels ${\rm G}_{i,j}$ to denote the configurations in Table $4$ corresponding to geometric subgroups. Here the first subscript $i$ indicates the specific $\C_i$ collection containing $H$ (see Table \ref{t:subs}). We refer the reader to Tables 3.2, 4.2 and 6.2 in \cite{BGT} for the cases where $H \in \C_1 \cup \C_3 \cup \C_6$, $H \in \C_2$ and $H \in \C_4$, respectively. Note that in the cases labelled ${\rm G}_{4,j}$ with $j \in \{2,3,4,5,7,8\}$ we write $H^{\circ} = A_1^m$, rather than $B_1^m$ or $C_1^m$ as given in \cite{BGT}. 

\item[(e)] In the statement of Theorem \ref{t:main_disc} we include the case $(G,p)=(B_n,2)$, which was excluded in the main theorem of \cite{BGT} on maximal geometric subgroups (note that $G$ acts reducibly on $W$ in this situation, so it does not arise in the study of non-geometric subgroups in \cite{BGMT}). As noted in \cite[Remark 1(b)]{BGT}, the triples that arise in this setting are easily obtained by inspecting the corresponding list of cases in \cite[Table 1]{BGT} for the group $C_n$. 

\vspace{2mm}

In fact, $D_n.2 < C_n$ with $\l = \l_n$ is the only relevant configuration in \cite[Table 1]{BGT}, where $D_n.2$ is the normalizer of the short root subsystem subgroup of type $D_n$. Here the connected component $D_n$ has two composition factors, with highest weights $2\delta_{n-1}$ and $2\delta_{n}$ (see \cite[Lemma 3.2.7]{BGT}). Since the support of $\l$ is on the long simple root, Lemma \ref{l:images} implies that the image of $C_n$ under the representation is of type $B_n$, and the image of the $D_n.2$ subgroup is the normalizer of
the long root subsystem subgroup of type $D_n$ in $B_n$. It is straightforward to see that this subgroup acts irreducibly on the spin module for $B_n$, and the composition factors for $D_n$ are the two spin modules. This configuration is recorded as case ${\rm G}_{6,2}$ in Table $4$.

\item[(f)] In presenting the examples in Table $4$, we adopt the same convention as per Table $3$, in that we only record cases up to duals (see Remark \ref{r:seitz}(c)). As noted in Remark \ref{r:seitz}(d), the situation where $G = D_{2m}$ and the highest weight of $V$ is not fixed by the involutory symmetry $\tau$ of the Dynkin diagram of $G$ requires special attention. In Table $4$, the relevant cases are as follows:
\[
{\rm G}_{1,4},\, {\rm G}_{1,5},\, {\rm G}_{2,4},\, {\rm G}_{2,5},\, {\rm G}_{2,6},\, {\rm G}_{2,7},\, {\rm G}_{3,1},\, {\rm G}_{4,5},\, {\rm G}_{4,6},\, {\rm G}_{4,7}.
\]
By carefully inspecting the appropriate tables in \cite{BGT}, we see that in each of these cases $H$ acts irreducibly on both $L_G(\l)$ and $L_G(\tau(\l))$, with the exception of the configurations 
labelled ${\rm G}_{3,1}$, ${\rm G}_{4,5}$, ${\rm G}_{4,6}$ and ${\rm G}_{4,7}$. In these four cases, the given subgroup $H$ acts irreducibly on exactly one of the 
$KG$-modules $L_G(\lambda)$ and $L_G(\tau(\lambda))$, and the image of $H$ under a graph automorphism of $G$ acts irreducibly on the other module; see \cite{BGT} for the details. As before, we will highlight these special cases by writing $(\dagger)$ in the second column of Table $4$. Note that in the cases where $H$ acts irreducibly on both modules, we just list the module $V = L_G(\l)$ with $\la \l, \a_{2m-1} \ra \ne 0$. Again, this is consistent with  the set-up in Table $3$.

\item[(g)] Finally, let us observe that the final column in Table $4$ gives the number $\kappa \geqs 2$ of $KH^{\circ}$-composition factors of $V|_{H^{\circ}}$. 
\end{itemize}
\end{re}

\subsection{Geometric subgroups}\label{ss:geom}

In this section, we discuss the work of Burness, Ghandour and Testerman in \cite{BGT}, which establishes Theorem \ref{t:main_disc} in the case where $H$ is a positive dimensional geometric subgroup of $G = Cl(W)$ in one of the $\C_i$ collections in Table \ref{t:subs}. 
 
Let $G = Cl(W)$, $H$ and $V = L_G(\l)$ be as in Hypothesis $(\star)$, and let us assume $H$ is a geometric subgroup of $G$. Recall that this means that $H$ is the stabilizer of a natural geometric structure on $W$ (such as a subspace, or a direct sum decomposition of $W$). Typically, this concrete description of $H$ allows us to study the restriction of a $T$-weight for $V$ (where $T$ is a maximal torus of $G$) to a suitably chosen maximal torus of the semisimple group $[H^{\circ},H^{\circ}]$. Moreover,
we may choose an appropriate Borel subgroup of $H$, containing this maximal torus, so that 
$\lambda|_{H^\circ}$ affords the highest weight of an $H^\circ$-composition factor of $V$.
Now if we assume that a certain coefficient in $\l$ is non-zero, then we can often identify a highest weight $\nu \ne \l|_{H^{\circ}}$ of a second $KH^{\circ}$-composition factor of $V$. At this point, we turn to the irreducibility of $V|_{H}$, which implies, via Clifford theory, that $\nu$ is conjugate to $\l|_{H^{\circ}}$ (under the action of the component group $H/H^{\circ}$). If we can find an incompatible highest weight $\nu$ (that is, one which is not conjugate to $\l|_{H^{\circ}}$), then this implies that the specific coefficient in $\l$ we are considering must be zero. In turn, we can then use this information to successively inspect other coefficients of $\l$, with the ultimate aim of identifying the precise $KG$-modules $V$ with $V|_{H}$ irreducible. 

In many ways, this is similar to the approach Ghandour uses in \cite{g_paper} to study the irreducible triples involving disconnected subgroups of exceptional algebraic groups. However, the problem for a classical group $G$ is more complicated because the rank of $G$ can be arbitrarily large, which means that a rather delicate combinatorial analysis of weights and restrictions is required in some cases.

\begin{re}\label{re:ford}
The analysis in \cite{BGT} of the irreducible triples $(G,H,V)$ with $G=B_n$ and $H=D_n.2$, as in Example \ref{ex:ford}, relies on Ford's earlier work in \cite{Ford1}. Indeed, in the proof of \cite[Lemma 3.2.1]{BGT} we reduce to the case where the composition factors of $V|_{H^{\circ}}$ are $p$-restricted, which allows us to apply Ford's result. This is the only part of the proof of the main theorem of \cite{BGT} which is not self-contained.
\end{re}

\begin{ex}\label{ex:c2}
Suppose $G  = {\rm SL}(W) = {\rm SL}_{8}(K) =  \langle U_{\pm \a_{i}} \mid 1 \leqs i \leqs 7 \rangle$, where $\Pi = \{\a_1, \ldots, \a_7\}$ is a set of simple roots for $G$ (with respect to a fixed Borel subgroup $B$ containing a maximal torus $T$ of $G$) and $U_{\a}$ denotes the 
$T$-root subgroup of $G$ corresponding to $\a \in \Sigma(G)$. Let $H$ be the stabilizer of a decomposition $W = W_1 \oplus W_2$, where $\dim W_i = 4$, so in terms of Theorem \ref{t:ls}, $H$ is a maximal geometric subgroup in the collection $\C_2$.  Set $X = [H^{\circ}, H^{\circ}]$ and let $S$ be a maximal torus of $X$. Then $H = H^{\circ}.2$ and without loss of generality, we may assume that 
$$H^{\circ}  = \left\{ \left(\begin{array}{c|c} x & \\ \hline & y \end{array}\right) \, : \, x,y \in {\rm GL}_{4}(K) \right\} \cap G$$
and
$$X = \la U_{\pm\a_1}, U_{\pm\a_2}, U_{\pm\a_3} \ra \times \la U_{\pm\a_5}, U_{\pm\a_6}, U_{\pm\a_7} \ra  \cong {\rm SL}_{4}(K) \times {\rm SL}_{4}(K).$$
We may also assume that $S<T$, so we can study the restriction of $T$-weights to $S$. 
Let $\{\l_1, \ldots, \l_7\}$ be the fundamental dominant weights for $G$ corresponding to the simple roots $\Pi$ above. From the structure of $X$, it follows that the restriction of a $T$-weight $\chi = \sum_{i}b_i\l_i$ to $S$ is given by
\begin{align*}
\chi|_{S}: & \begin{tikzpicture}[scale=.4, baseline=(current bounding box.base)]
   \tikzstyle{every node}=[font=\small]
   \draw (0,1.5) node[anchor=north] {$b_1$};
   \draw (2,1.5) node[anchor=north] {$b_2$};
   \draw (4,1.5) node[anchor=north] {$b_3$};
      \foreach \x in {0,...,2}
    \draw[xshift=\x cm,fill=black] (\x cm,0) circle (.25cm);
    \foreach \y in {0.1,...,1.1}
    \draw[xshift=\y cm] (\y cm,0) -- +(1.6cm,0);
     \end{tikzpicture} \oplus \begin{tikzpicture}[scale=.4, baseline=(current bounding box.base)]
   \tikzstyle{every node}=[font=\small]
   \draw (0,1.5) node[anchor=north] {$b_5$};
   \draw (2,1.5) node[anchor=north] {$b_6$};
   \draw (4,1.5) node[anchor=north] {$b_7$};
      \foreach \x in {0,...,2}
    \draw[xshift=\x cm,fill=black] (\x cm,0) circle (.25cm);
    \foreach \y in {0.1,...,1.1}
    \draw[xshift=\y cm] (\y cm,0) -- +(1.6cm,0);
     \end{tikzpicture}
     \end{align*}
Set $V=L_G(\l)$, where $\l = \sum_{i}a_i\l_i$ is $p$-restricted, and let us assume $V|_{H}$ is irreducible. 

\paragraph{Claim.} $V = W$ or $W^*$ are the only possibilities (that is, $\l = \l_1$ or $\l_7$).

\vs

To see this, we first apply the main theorem of \cite{Seitz2}, which implies that $V|_{H^{\circ}}$ is reducible and thus  $V|_{H^{\circ}} = V_1 \oplus V_2$ by Clifford theory, where $V_1|_{X} = L_X(\l|_{S})$ and 
$V_2|_{X} = L_X(\mu|_{S})$ with
\begin{equation}\label{e:chi}
\begin{split}
\l|_{S}: & \begin{tikzpicture}[scale=.4, baseline=(current bounding box.base)]
   \tikzstyle{every node}=[font=\small]
   \draw (0,1.3) node[anchor=north] {$a_1$};
   \draw (2,1.3) node[anchor=north] {$a_2$};
   \draw (4,1.3) node[anchor=north] {$a_3$};
      \foreach \x in {0,...,2}
    \draw[xshift=\x cm,fill=black] (\x cm,0) circle (.25cm);
    \foreach \y in {0.1,...,1.1}
    \draw[xshift=\y cm] (\y cm,0) -- +(1.6cm,0);
     \end{tikzpicture} \oplus \begin{tikzpicture}[scale=.4, baseline=(current bounding box.base)]
   \tikzstyle{every node}=[font=\small]
   \draw (0,1.3) node[anchor=north] {$a_5$};
   \draw (2,1.3) node[anchor=north] {$a_6$};
   \draw (4,1.3) node[anchor=north] {$a_7$};
      \foreach \x in {0,...,2}
    \draw[xshift=\x cm,fill=black] (\x cm,0) circle (.25cm);
    \foreach \y in {0.1,...,1.1}
    \draw[xshift=\y cm] (\y cm,0) -- +(1.6cm,0);
     \end{tikzpicture} \\
\mu|_{S}: & \begin{tikzpicture}[scale=.4, baseline=(current bounding box.base)]
   \tikzstyle{every node}=[font=\small]
   \draw (0,1.3) node[anchor=north] {$a_5$};
   \draw (2,1.3) node[anchor=north] {$a_6$};
   \draw (4,1.3) node[anchor=north] {$a_7$};
      \foreach \x in {0,...,2}
    \draw[xshift=\x cm,fill=black] (\x cm,0) circle (.25cm);
    \foreach \y in {0.1,...,1.1}
    \draw[xshift=\y cm] (\y cm,0) -- +(1.6cm,0);
     \end{tikzpicture} \oplus \begin{tikzpicture}[scale=.4, baseline=(current bounding box.base)]
   \tikzstyle{every node}=[font=\small]
   \draw (0,1.3) node[anchor=north] {$a_1$};
   \draw (2,1.3) node[anchor=north] {$a_2$};
   \draw (4,1.3) node[anchor=north] {$a_3$};
      \foreach \x in {0,...,2}
    \draw[xshift=\x cm,fill=black] (\x cm,0) circle (.25cm);
    \foreach \y in {0.1,...,1.1}
    \draw[xshift=\y cm] (\y cm,0) -- +(1.6cm,0);
     \end{tikzpicture}
\end{split}
\end{equation}
(recall that $\mu|_{S}$ is conjugate to $\l|_{S}$, under the induced action of the component group $H/H^{\circ}$). In particular, $V|_{H^{\circ}}$ has exactly two composition factors, with highest weights $\l|_{S}$ and $\mu|_{S}$ as above.

First assume $a_4 \ne 0$ and let $s_4$ be the fundamental reflection in the Weyl group $W(G)$ corresponding to $\a_4$. Under the action of $W(G)$ on the set $\L(V)$ of weights of $V$, we have $s_4(\l) = \l - a_4\a_4 \in \L(V)$ and we deduce that 
$$\chi = \l - \a_4  = \l + \l_3-2\l_4+\l_5 \in \L(V)$$ 
(see \cite[1.30]{Test1}). If $v$ is a non-zero vector in the weight space $V_{\chi}$, then one checks that $v$ is a maximal vector for the Borel subgroup $B_{H^{\circ}} = \langle T, U_{\a_i} \mid i \ne 4\rangle$ of $H^{\circ}$ (this follows from \eqref{e:mt}) and thus $\chi$ affords the highest weight of a $KH^{\circ}$-composition factor of $V$. In particular, we must have $\chi|_{S} = \l|_{S}$ or $\mu|_{S}$, but this is not possible since 
\begin{align*}
\chi|_{S}: & \begin{tikzpicture}[scale=.4, baseline=(current bounding box.base)]
   \tikzstyle{every node}=[font=\small]
   \draw (0,1.4) node[anchor=north] {$a_1$};
   \draw (2,1.4) node[anchor=north] {$a_2$};
   \draw (4,1.55) node[anchor=north] {$a_3+1$};
      \foreach \x in {0,...,2}
    \draw[xshift=\x cm,fill=black] (\x cm,0) circle (.25cm);
    \foreach \y in {0.1,...,1.1}
    \draw[xshift=\y cm] (\y cm,0) -- +(1.6cm,0);
     \end{tikzpicture} \oplus \begin{tikzpicture}[scale=.4, baseline=(current bounding box.base)]
   \tikzstyle{every node}=[font=\small]
   \draw (0,1.55) node[anchor=north] {$a_5+1$};
   \draw (2,1.4) node[anchor=north] {$a_6$};
   \draw (4,1.4) node[anchor=north] {$a_7$};
      \foreach \x in {0,...,2}
    \draw[xshift=\x cm,fill=black] (\x cm,0) circle (.25cm);
    \foreach \y in {0.1,...,1.1}
    \draw[xshift=\y cm] (\y cm,0) -- +(1.6cm,0);
     \end{tikzpicture}
     \end{align*}
We conclude that $a_4=0$.

Similarly, if we now assume $a_3 \ne 0$ then $\chi = \l - \a_3-\a_4$ affords the highest weight of a $KH^{\circ}$-composition factor of $V|_{H^{\circ}}$ and once again we reach a contradiction since 
\begin{align*}
\chi|_{S}: & \begin{tikzpicture}[scale=.4, baseline=(current bounding box.base)]
   \tikzstyle{every node}=[font=\small]
   \draw (0,1.4) node[anchor=north] {$a_1$};
   \draw (2,1.55) node[anchor=north] {$a_2+1$};
   \draw (4.5,1.5) node[anchor=north] {$a_3-1$};
      \foreach \x in {0,...,2}
    \draw[xshift=\x cm,fill=black] (\x cm,0) circle (.25cm);
    \foreach \y in {0.1,...,1.1}
    \draw[xshift=\y cm] (\y cm,0) -- +(1.6cm,0);
     \end{tikzpicture} \oplus \begin{tikzpicture}[scale=.4, baseline=(current bounding box.base)]
   \tikzstyle{every node}=[font=\small]
   \draw (0,1.55) node[anchor=north] {$a_5+1$};
   \draw (2,1.4) node[anchor=north] {$a_6$};
   \draw (4,1.4) node[anchor=north] {$a_7$};
      \foreach \x in {0,...,2}
    \draw[xshift=\x cm,fill=black] (\x cm,0) circle (.25cm);
    \foreach \y in {0.1,...,1.1}
    \draw[xshift=\y cm] (\y cm,0) -- +(1.6cm,0);
     \end{tikzpicture}
     \end{align*}
     is not equal to $\l|_{S}$ nor $\mu|_{S}$. Continuing in this way, we reduce to the case $\l = a_1\l_1+a_7\l_7$. 

If $a_1 \ne 0$ then one checks that $\chi = \l - \a_1-\a_2-\a_3-\a_4$ affords the highest weight of a composition factor of $V|_{H^{\circ}}$, so 
\begin{align*}
\chi|_{S}: & \begin{tikzpicture}[scale=.4, baseline=(current bounding box.base)]
   \tikzstyle{every node}=[font=\small]
   \draw (0,1.4) node[anchor=north] {$a_1-1$};
   \draw (2,1.4) node[anchor=north] {0};
   \draw (4,1.4) node[anchor=north] {0};
      \foreach \x in {0,...,2}
    \draw[xshift=\x cm,fill=black] (\x cm,0) circle (.25cm);
    \foreach \y in {0.1,...,1.1}
    \draw[xshift=\y cm] (\y cm,0) -- +(1.6cm,0);
     \end{tikzpicture} \oplus \hspace{-0.7mm} \begin{tikzpicture}[scale=.4, baseline=(current bounding box.base)]
   \tikzstyle{every node}=[font=\small]
   \draw (0,1.4) node[anchor=north] {$1$};
   \draw (2,1.4) node[anchor=north] {0};
   \draw (4,1.4) node[anchor=north] {$a_7$};
      \foreach \x in {0,...,2}
    \draw[xshift=\x cm,fill=black] (\x cm,0) circle (.25cm);
    \foreach \y in {0.1,...,1.1}
    \draw[xshift=\y cm] (\y cm,0) -- +(1.6cm,0);
     \end{tikzpicture}
     \end{align*}
     must be equal to 
\begin{align*}
\mu|_{S}: & \; \; \begin{tikzpicture}[scale=.4, baseline=(current bounding box.base)]
   \tikzstyle{every node}=[font=\small]
   \draw (0,1.4) node[anchor=north] {0};
   \draw (2,1.4) node[anchor=north] {0};
   \draw (4,1.3) node[anchor=north] {$a_7$};
      \foreach \x in {0,...,2}
    \draw[xshift=\x cm,fill=black] (\x cm,0) circle (.25cm);
    \foreach \y in {0.1,...,1.1}
    \draw[xshift=\y cm] (\y cm,0) -- +(1.6cm,0);
     \end{tikzpicture} \hspace{-0.7mm} \oplus \hspace{-0.7mm} \begin{tikzpicture}[scale=.4, baseline=(current bounding box.base)]
   \tikzstyle{every node}=[font=\small]
   \draw (0,1.3) node[anchor=north] {$a_1$};
   \draw (2,1.4) node[anchor=north] {0};
   \draw (4,1.4) node[anchor=north] {0};
      \foreach \x in {0,...,2}
    \draw[xshift=\x cm,fill=black] (\x cm,0) circle (.25cm);
    \foreach \y in {0.1,...,1.1}
    \draw[xshift=\y cm] (\y cm,0) -- +(1.6cm,0);
     \end{tikzpicture}
     \end{align*}
and thus $(a_1,a_7) = (1,0)$ is the only possibility. Therefore, $\l = \l_1$ and $V=W$ is the natural module for $G$. Similarly, if we assume $a_7 \ne 0$ then we deduce that $(a_1,a_7) = (0,1)$ and $V = W^*$ is the dual of the natural module (of course, $H$ does act irreducibly on both $W$ and $W^*$, so the two possibilities we have ended up with are genuine examples).
\end{ex}

To conclude, let us clarify the status of the main theorem of \cite{BGT} in light of the omissions which have recently been identified in Seitz's main theorem in \cite{Seitz2}.

\begin{prop}\label{p:geomcheck}
Let  $G,H,V$ be given as in Hypothesis $(\star)$ and let $\varphi:G \to {\rm SL}(V)$ be the corresponding representation. In addition, let us assume $H$ is a geometric subgroup of $G$. Then $V|_{H}$ is irreducible if and only if $(\varphi(G),\varphi(H),V)$ is one of the cases labelled ${\rm G}_{i,j}$ in Table $4$. 
\end{prop}

\begin{proof} 
This follows directly from the proof of the main theorem in \cite{BGT}. As indicated in the discussion preceding Example \ref{ex:c2}, we work directly with an explicit description of each geometric maximal subgroup $H$, which allows us to study the embedding of a maximal torus of $H$ in a maximal torus of $G$. In addition, in some cases we even have a concrete description of the root groups of $H$ in terms of the root groups of $G$. As noted in Remark \ref{re:ford}, we  do use Ford's work in \cite{Ford1} to handle the case labelled ${\rm G}_{1,1}$ in Table $4$. Here $G = B_n$, $H = D_n.2$ and Ford works with a precise description of the embedding, through a series of calculations with the corresponding universal enveloping algebras. Therefore, in all cases, the existence of the new examples discussed in Section \ref{ss:new} (see the cases labelled ${\rm IV}_1'$ and ${\rm S}_{10}$ in Table $3$) has no impact on the proof of the main theorem in \cite{BGT}.
\end{proof}

\subsection{Non-geometric subgroups}\label{ss:ngeom}

Here we briefly discuss the proof of Theorem \ref{t:main_disc} in the remaining case where $H$ is a disconnected non-geometric almost simple subgroup of $G = Cl(W)$ in the collection $\mathcal{S}$. The proof is due to Burness, Ghandour, Marion and Testerman (see \cite{BGMT}).

Set $X = H^{\circ}$. Since $H$ is disconnected, it follows that $X$ is of type $A_m$ (with $m \geqs 2$), $D_m$ (with $m \geqs 4$) or $E_6$. In addition, $W$ is an irreducible and tensor indecomposable $KX$-module, so we may write $W|_{X}=L_X(\delta)$. Also recall from Theorem \ref{t:ls} that if $G = {\rm SL}(W)$ then $X$ does not fix a nondegenerate form on $W$. In addition, following \cite{BGMT}, we may assume that $X$ satisfies the conditions
\begin{itemize}
\item[(I)] If $(G,p) = (C_n,2)$, then $X$ does not fix a nondegenerate quadratic form on $W$ (see Remark \ref{r:sdef}), and 
\item[(II)] $W=L_X(\delta)$ is a $p$-restricted $KX$-module. 
\end{itemize}
Note that if $(G,p)=(B_n,2)$ then $G$ acts reducibly on $W$, so this case does not arise in the study of irreducible triples with $H \in \mathcal{S}$. 

The next result is the main theorem of \cite{BGMT} (in the statement, we write $\{\delta_1, \ldots, \delta_m\}$ for a set of fundamental dominant weights for $H^{\circ}$, labelled in the usual way). The example appearing in the statement is recorded as case ${\rm N}_{1}$ in Table $4$.

\begin{thm}\label{t:ngeom}
Let $G$, $H$ and $V = L_G(\l)$ be given as in Hypothesis $(\star)$ and assume $H \in \mathcal{S}$ satisfies conditions {\rm (I)} and {\rm (II)} above. Then $V|_{H}$ is irreducible if and only if 
$G = C_{10}$, $H = A_5.2$ and $\l=\l_3$, with $p \neq 2,3$ and $W|_{H^{\circ}} = L_{H^{\circ}}(\delta_3)$.
\end{thm}

\begin{re}
Let $(G,H,V)$ be the irreducible triple arising in the statement of Theorem \ref{t:ngeom}. Here 
$$V|_{H^{\circ}} = V_1 \oplus V_2 =  \begin{tikzpicture}[scale=.4, baseline=-0.8mm]
   \draw (0,1.4) node[anchor=north] {$1$};
   \draw (6,1.4) node[anchor=north] {$2$};
      \tikzstyle{every node}=[font=\small]
    \foreach \x in {0,...,4}
    \draw[xshift=\x cm,fill=black] (\x cm,0) circle (.25cm);
    \foreach \y in {0.1,...,3.1}
    \draw[xshift=\y cm] (\y cm,0) -- +(1.6 cm,0);
  \end{tikzpicture} \hspace{1mm} \oplus \hspace{1mm} \begin{tikzpicture}[scale=.4, baseline=-0.8mm]
   \draw (2,1.4) node[anchor=north] {$2$};
   \draw (8,1.4) node[anchor=north] {$1$};
   \tikzstyle{every node}=[font=\small]
    \foreach \x in {0,...,4}
    \draw[xshift=\x cm,fill=black] (\x cm,0) circle (.25cm);
    \foreach \y in {0.1,...,3.1}
    \draw[xshift=\y cm] (\y cm,0) -- +(1.6 cm,0);
  \end{tikzpicture}$$
and thus the two $KH^{\circ}$-composition factors of $V$ are $p$-restricted. In particular, $(G,H,V)$ satisfies the conditions on the irreducible triples studied by Ford in \cite{Ford1,Ford2}. However, in the statement of Ford's main theorem, this configuration has been omitted in error. We refer the reader to \cite[Remark 3.6.18]{BGMT} for further details.
\end{re}

As described in Section \ref{ss:geom}, we can study the irreducible triples of the form $(G,H,V)$ for a geometric subgroup $H$ by considering weights and their restrictions to an appropriate maximal torus of $[H^{\circ}, H^{\circ}]$, which relies on the fact that we have a concrete description of the embedding of $H$ in $G$ in terms of root subgroups. However, a non-geometric subgroup $H \in \mathcal{S}$ is embedded via an arbitrary $p$-restricted, tensor indecomposable representation of $X=H^{\circ}$, which renders the previous approach infeasible. Therefore, in order to handle this situation we turn to the approach used by Seitz in \cite{Seitz2} (and also Ford \cite{Ford1,Ford2}), which relies on studying the action of certain parabolic subgroups of $G$ on $V$, which are constructed in a natural way from parabolic subgroups of $X$. 

For example, suppose $H = X\langle t \rangle$, where $t$ is an involutory graph automorphism of $X$. Let us assume $V|_{H}$ is irreducible, with $V|_X = V_1 \oplus V_2$, and let $P_X=Q_XL_X$ be a parabolic subgroup of $X$, where $Q_X$ is the unipotent radical and $L_X$ is a Levi factor. Following Seitz \cite{Seitz2} and Ford \cite{Ford1,Ford2}, we consider the flag
\begin{equation}\label{e:flag}
W > [W,Q_X^1] > [W,Q_X^2] > \cdots > 0,
\end{equation}
where $[W,Q_X^0] = W$ and 
$$[W,Q_X^i]=\la qw-w \,:\, w \in [W,Q_X^{i-1}], \, q \in Q_X\ra$$
for $i \geqs 1$. Set $W_i = [W,Q_{X}^{i}] / [W,Q_{X}^{i+1}]$ for each $i \geqs 0$ and let $P=QL$ be the stabilizer in $G$ of this flag, which is a parabolic subgroup of $G$ with unipotent radical $Q$ and Levi factor $L$, with several desirable properties. In particular, $Q_X<Q$ (see \cite[Lemma 2.7.1]{BGMT}). 

We can study the weights occurring in $[W,Q_X^i]$ in order to obtain a lower bound on $\dim W_i$, which then leads to structural information on $L'$.  Now a theorem of Smith \cite{Smith} implies that $L'$ acts irreducibly on the quotient $V/[V,Q]$, and similarly $L_X'$ on each $V_i/[V_i,Q_X]$. By combining this observation with our lower bounds on the dimensions of the quotients in \eqref{e:flag}, we can obtain restrictions on the highest weight $\l$ of $V$. 
For instance, if we consider the case where $P_X$ is a $t$-stable Borel subgroup of $X$, then \cite[Lemma 5.1]{Ford1} implies that $L'$ has  an $A_1$ factor and this severely restricts the coefficients of $\l$ corresponding to the roots in $L'$ (this is illustrated in Example \ref{ex:ng} below). 

Beyond the choice of a $t$-stable Borel subgroup, we often study the embedding of other
parabolic subgroups $P_X$ of $X$ in parabolic subgroups of $G$, again with an inclusion of unipotent radicals as above. In the general setting, both $L'$ and $L_X'$ act on the space $V/[V,Q]$, and as a module for $L_X'$ we have
\[
V/[V,Q] = V_1/[V_1,Q_X]\oplus V_2/[V_2,Q_X] \mbox{ or } V_i/[V_i,Q_X]
\]
for $i=1,2$. If we write $L'=L_1\cdots L_k$, a commuting product of simple algebraic groups,  then 
\[
V/[V,Q]=M_1\otimes\cdots\otimes  M_k,
\] 
where each $M_i$ is an irreducible $KL_i$-module. In particular, if $M_i$ is nontrivial then we  may consider the triple $(L_i,\pi_i(L_X'),M_i)$,
where $\pi_i:L_X'\to L_i$ is the natural projection map induced by the isomorphisms and inclusions 
\[
L_X\cong P_X/Q_X\cong P_XQ/Q< P/Q\cong L.
\] 
It follows that $(L_i,\pi_i(L_X'),M_i)$ is either an irreducible triple, or $L_X'$ has precisely two composition factors on $M_i$. This allows us to apply induction on the rank of $G$, using the previously considered smaller rank cases to obtain information on the highest weights of the $M_i$, and ultimately to determine the highest weight $\lambda$ of $V$.

\begin{ex}\label{ex:ng}
Suppose $H = A_2\langle t \rangle = A_2.2$ and $p \ne 2,5$, so $X=A_2$ and $t$ is an involutory graph automorphism of $X$. Let $\{\delta_1, \delta_2\}$ be the fundamental dominant weights for $X$ corresponding to the simple roots $\{\b_1,\b_2\}$. Set $W=L_X(\delta)$, where $\delta = 2\delta_1+2\delta_2$, so $\dim W = 27$ and the action of $X$ on $W$ embeds $X$ in $G = {\rm SO}(W)$. Moreover, the symmetry of $\delta$ implies that $W$ is self-dual, so $t$ also acts on $W$ and thus $H<G$. Set $V=L_G(\l)$, where $\l = \sum_{i}a_i\l_i$ is $p$-restricted and $\l \ne \l_1$.

\paragraph{Claim.} $V|_{H}$ is reducible.

\vs

Seeking a contradiction, let us assume $V|_{H}$ is irreducible. By the main theorem of \cite{Seitz2}, $V|_{X}$ is reducible and so Clifford theory implies that $V|_{X} = V_1 \oplus V_2$, where $V_i = L_X(\mu_i)$. Without loss of generality, we may assume that $\mu_1 =\l|_{X} = c_1\delta_1+c_2\delta_2$, in which case $\mu_2$ is the image of $\mu_1$ under the action of $t$, so $\mu_2 = c_2\delta_1+c_1\delta_2$. Moreover, the irreducibility of $V|_{H}$ implies that $V_1$ and $V_2$ are non-isomorphic (see \cite[Proposition 2.6.2]{BGMT}), whence $c_1 \ne c_2$.

Let $P_X=Q_XL_X$ be a $t$-stable Borel subgroup of $X$ (so $L_X=T_X$ is a maximal torus of $X$) and consider the corresponding flag of $W$ given in \eqref{e:flag}. Define the quotients $W_i$ as above. Using \cite[(2.3)]{Seitz2}, one checks that
$$\dim W_0 = 1, \; \dim W_1 = 2, \; \dim W_2 = 4, \; \dim W_3 = 4, \; \dim W_4 = 5,$$
which implies that $L'=A_1A_3A_3B_2$. By the aforementioned theorem of Smith \cite{Smith}, 
$$V/[V,Q] = V_1/[V_1,Q_X] \oplus V_2/[V_2,Q_X] = M_1 \otimes M_2 \otimes M_3 \otimes M_4$$
is an irreducible $KL'$-module, with highest weight $\l|_{L'}$ (here the notation indicates that $M_i$ is an irreducible module for the $i$-th factor of $L'$). But $V/[V,Q]$ is $2$-dimensional, so $M_1$ must be the natural module for the $A_1$ factor of $L'$, and $M_2$, $M_3$ and $M_4$ are trivial (see \cite[Lemma 5.2]{Ford1}). This observation yields useful information on the coefficients in $\l$; indeed, we obtain the following partially labelled diagram:
 $$
 
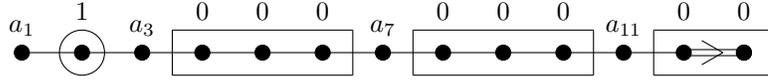
\begin{tikzpicture}[scale=.4, baseline=(current bounding box.base)]
   \tikzstyle{every node}=[font=\small]
    \foreach \x in {0,...,12}
    \draw[xshift=\x cm,fill=black] (\x cm,0) circle (.25cm);
    \draw[xshift=12 cm,fill=black] (12 cm, 0) circle (.25cm);
    \foreach \y in {0.1,...,10.1}
    \draw[xshift=\y cm] (\y cm,0) -- +(1.6cm,0);
    \draw (22.2 cm, .1 cm) -- +(1.6 cm,0);
    \draw (22.2 cm, -.1 cm) -- +(1.6 cm,0);
    \draw (23.3,0) --+ (-0.7,0.4);
    \draw (23.3,0) --+ (-0.7,-0.4);
    \draw (2 cm, 0) circle (.75cm);
    \draw (0cm,1.4) node[anchor=north] {$a_1$};
    \draw (4cm,1.4) node[anchor=north] {$a_3$};
    \draw (12cm,1.4) node[anchor=north] {$a_7$};
    \draw (20cm,1.4) node[anchor=north] {$a_{11}$};
    \draw (2cm,2) node[anchor=north] {$1$};
     \draw (6cm,2) node[anchor=north] {$0$};
      \draw (8cm,2) node[anchor=north] {$0$};
       \draw (10cm,2) node[anchor=north] {$0$};
       \draw (14cm,2) node[anchor=north] {$0$};
      \draw (16cm,2) node[anchor=north] {$0$};
       \draw (18cm,2) node[anchor=north] {$0$};
       \draw (22cm,2) node[anchor=north] {$0$};
       \draw (24cm,2) node[anchor=north] {$0$};
    \draw (5cm,-0.75cm) -- (5cm,0.75cm);
    \draw (5cm,0.75cm) -- (11cm,0.75cm);
     \draw (11cm,0.75cm) -- (11cm,-0.75cm);
      \draw (11cm,-0.75cm) -- (5cm,-0.75cm);
      \draw (13cm,-0.75cm) -- (13cm,0.75cm);
    \draw (13cm,0.75cm) -- (19cm,0.75cm);
     \draw (19cm,0.75cm) -- (19cm,-0.75cm);
      \draw (19cm,-0.75cm) -- (13cm,-0.75cm);
      \draw (21cm,-0.75cm) -- (21cm,0.75cm);
    \draw (21cm,0.75cm) -- (25cm,0.75cm);
     \draw (25cm,0.75cm) -- (25cm,-0.75cm);
      \draw (25cm,-0.75cm) -- (21cm,-0.75cm);
           \end{tikzpicture}$$
           In other words, we have reduced to the case where 
           \begin{equation}\label{e:lam}
           \l = a_1\l_1+\l_2 + a_3\l_3+a_7\l_7 +a_{11}\l_{11}.
           \end{equation}

By choosing an appropriate ordering on the $T_X$-weights in $W$, we can determine the restrictions of the simple roots $\a_i$ to the subtorus $T_X$ (see \cite[Table 3.2]{BGMT}). For example, we can fix an ordering so that $\a_2|_{X} = \b_2-\b_1$. Now, since the coefficient of $\l_2$ in \eqref{e:lam} is positive, it follows that  $\l - \a_2$ is a $T$-weight of $V$ and thus $(\l-\a_2)|_{X} = \mu_1 - \b_2+\b_1$ is a $T_X$-weight of $V|_{X}$. Moreover, by considering root restrictions we can show that this is the highest weight of a $KX$-composition factor of $V$, so we must have $\mu_2 = \mu_1 - \b_2+\b_1$. 

By \cite[(8.6)]{Seitz2}, $\delta-\b_1-\b_2$ has multiplicity $2$ as a $T_X$-weight of $W$ (here we are using the fact that $p \ne 5$), which means that we may assume $(\l - \sum_{i=1}^{6}\a_i)|_{X} =  \delta-\b_1-\b_2$ and thus $\a_6|_{X}=0$. It follows that there are at least three $T$-weights of $V$ which give $\nu = \l|_{X}-\b_2$ on restriction to $T_X$, so the multiplicity of $\nu$ is at least $3$. However, one checks that $\nu$ has multiplicity $1$ in both $V_1$ and $V_2$, so this is incompatible with the decomposition $V|_{X} = V_1 \oplus V_2$. This contradiction completes the proof of the claim.
\end{ex}

Finally, let us address the veracity of the main theorem of \cite{BGMT}, in view of the corrections to Seitz's main theorem discussed in Section \ref{ss:new}. 

\begin{prop}\label{p:nongeomcheck} 
Let  $G,H,V$ be given as in Hypothesis $(\star)$ and let us assume $H$ is non-geometric.
Then $V|_{H}$ is irreducible if and only if $(G,H,V)$ is as given in Theorem $\ref{t:ngeom}$.
\end{prop}

\begin{proof} 
As discussed above, in the proof of the main result of \cite{BGMT} we proceed by induction on the rank of $G$, studying the embedding of various parabolic subgroups of $H$ in parabolic subgroups of $G$. We then rely upon an inductive list of examples, for disconnected as well as connected subgroups. In view of the recently discovered configurations described in Section \ref{ss:new}, we must go back through our proof, paying particular attention to all cases where the Levi factor of the chosen parabolic of $H^{\circ}$ has type $B_{\ell}$. But for the groups $H$ satisfying the hypotheses of the proposition, we have $H^{\circ} = A_m$ (with $m\geqs 2$), $D_m$ (with $m\geqs 4$) or $E_6$. In particular, $H^{\circ}$ does not have a parabolic subgroup whose Levi factor is of type $B_{\ell}$, so the arguments in \cite{BGMT} are not affected by the new examples.
\end{proof}

\section{Tables}\label{s:tab}

In this final section we present Tables $3$ and $4$, which describe the irreducible triples arising in the statements of Theorems \ref{t:seitz} and \ref{t:main_disc}. Note that Table $3$ can be viewed as a corrected version of \cite[Table 1]{Seitz2}. In the following two remarks, we explain the set-up and notation in each table.

\begin{re}\label{r:tab5}
Let us describe the notation in Table $3$. 
\begin{itemize}
\item[(a)] In the first column we give the label for the particular case in that row of the table; our choice of labels is consistent with \cite[Table 1]{Seitz2}. As explained in Remark \ref{r:seitz}, 
we have included the following additional cases in Table $3$:
\[
{\rm S}_{10}, \; {\rm S}_{11},\; {\rm MR}_{1}', \; {\rm MR}_{2}'
\]
and we have corrected the conditions in the case labelled ${\rm IV}_{1}'$.

\item[(b)] The second column gives the pair $(G,H)$, plus any additional information on ${\rm char}(K)=p$ and the rank of $G$, as needed. If no restrictions on $p$ are presented, then it is to be assumed that the only conditions on $p$ are to ensure that the highest weight of $V = L_G(\l)$ (as given in the final column of the table) is $p$-restricted. For example, in case ${\rm IV}_{1}$ we have $\l = k\l_{n}$, so this means that $p=0$ or $p \geqs k+1$. 

\item[(c)] In terms of the second column, the case labelled ${\rm IV}_{2}'$ requires special attention. Here $H$ is a subgroup of a central product $B_kB_{n-k}$, so there are natural projection maps 
$\pi_1:H\to B_k$ and $\pi_2:H \to B_{n-k}$. Seitz's notation  
\[
H \to' B_kB_{n-k}<D_{n+1}
\] 
indicates that either $H$ projects onto both simple factors, or some factor is of type $B_2$ and the projection of $H$ to this factor is a group of type $A_1$ acting irreducibly on the spin module for the $B_2$ (in the latter situation, note that we also need the condition $p \geqs 5$, as indicated in the table). In addition, as Seitz notes on \cite[p.9]{Seitz2}, it may be necessary for the projections to involve distinct field twists to ensure the irreducibility of $V|_{H}$. For instance, if $A_1 \to' B_2B_2<D_5$ then we need different field twists on the embedding of $A_1$ in each $B_2$ factor. This notation is also used in cases ${\rm S}_{6}$ and ${\rm MR}_{5}$, with the same meaning. 

\item[(d)] In case ${\rm IV}_{9}$ we write $(\dagger)$ in the second column to indicate that the given subgroup $H$ acts irreducibly on exactly one of the two spin modules for $G$ (see Remark \ref{r:seitz}(d)).

\item[(e)] In the case labelled ${\rm MR}_{1}$, we use the notation $\widetilde{A}_2$ to denote a subsystem subgroup of type $A_2$ corresponding to short roots in $G_2$ (and similarly $\widetilde{D}_{4}<F_4$ in case ${\rm MR}_{2}$). See Remark \ref{r:seitz}(g) for further details.

\item[(f)] The embedding of $H$ in $G$ is described in the third column of the table, in terms of the action of $H$ on the $KG$-module $W$ associated with $G$. In particular, if $G$ is of classical type then $W$ is just the natural module. Similarly, if $G$ has type $G_2$, $F_4$, $E_6$ or $E_7$ then $W$ is the minimal module for $G$ of dimension $7$ ($6$ if $p=2$), $26$ ($25$ if $p=3$), $27$ or $56$, respectively (note that the case $G = E_8$ does not arise in Table $3$). 

\vspace{1.5mm}

Typically, $H$ is simple and $W|_{H}$ is irreducible, in which case we give the highest weight of $W|_{H}$ in terms of a set of fundamental dominant weights $\{\delta_1, \ldots, \delta_m\}$ for $H$. This notation is suitably modified in cases ${\rm IV}_{3}$ and ${\rm IV}_{5}$, where $H = H_1H_2$ is a central product of two simple groups. Note that in cases ${\rm T}_1$ and 
${\rm T}_2$, we write $0$ for the highest weight of the trivial module. Following \cite[Table 1]{Seitz2}, in some cases we write ``usual" in the third column to indicate that the embedding of $H$ in $G$ is via the usual action of $H$ on the natural $KG$-module. For example, in the case labelled ${\rm IV}_{1}$, $H = B_n$ is embedded in $G = D_{n+1}$ as the stabilizer of a nonsingular $1$-space.

\item[(g)] In the fourth and fifth columns, we give the highest weights of $V|_{H}$ and $V$, respectively, in terms of labelled Dynkin diagrams, together with any additional conditions on the relevant parameters. Once again, this is consistent with \cite[Table 1]{Seitz2}, with the exception that the conditions in case ${\rm IV}_{1}'$ have been corrected in view of \cite{CT} (see Section \ref{ss:new}). 
\end{itemize}
\end{re}

\vspace{35pt}

\begin{center}
Table 3. The irreducible triples with $H$ connected
\end{center}

{\tiny
  \setlength{\tabcolsep}{4.96pt}
\begin{longtable}{lllll} 
no. & $H<G$ & $W|_{H}$ & $V|_{H}$ & $V$ \\ 
& & & & \\
${\rm I}_{1}$ & $C_n < A_{2n-1}$, $n \geqs 2, p \ne 2$ &  $\delta_1$ & \hspace{-1.65mm}
\begin{tikzpicture}[scale=.25, baseline=(current bounding box.base)]
   \tikzstyle{every node}=[font=\tiny]
   \draw (0,1.5) node[anchor=north] {$k$};
    \foreach \x in {0,...,2}
    \draw[xshift=\x cm,fill=black] (\x cm,0) circle (.25cm);
    \draw[xshift=3 cm,fill=black] (3 cm, 0) circle (.25cm);
    \draw[dotted,thick] (2.55cm,0) -- +(1cm,0);
    \foreach \y in {0.1,...,0.1}
    \draw[xshift=\y cm] (\y cm,0) -- +(1.6cm,0);
    \draw (4.2 cm, .1 cm) -- +(1.6 cm,0);
    \draw (4.2 cm, -.1 cm) -- +(1.6 cm,0);
    \draw (4.8,0) --+ (0.72,0.4);
    \draw (4.8,0) --+ (0.72,-0.4);
  \end{tikzpicture}
  & \hspace{-1.65mm}
  \begin{tikzpicture}[scale=.25, baseline=(current bounding box.base)]
   \tikzstyle{every node}=[font=\tiny]
   \draw (0,1.5) node[anchor=north] {$k$};
    \foreach \x in {0,...,3}
    \draw[xshift=\x cm,fill=black] (\x cm,0) circle (.25cm);
    \draw[dotted,thick] (2.55cm,0) -- +(1cm,0);
    \foreach \y in {0.1,...,0.1}
    \draw[xshift=\y cm] (\y cm,0) -- +(1.6 cm,0);
    \foreach \y in {2.1,...,2.2}
    \draw[xshift=\y cm] (\y cm,0) -- +(1.6 cm,0);
  \end{tikzpicture}
  \\
  & & & $k \geqs 2$ & \\
\end{longtable}}

\begin{re}\label{r:tab6}
Let us make some comments on the notation appearing in Table $4$.
\begin{itemize}
\item[(a)] The first column gives the label for each case $(G,H,V)$ in the table. The configurations ${\rm G}_{1,1}$-${\rm G}_{6,2}$ are cases where $H$ is a geometric maximal subgroup of $G$, whereas ${\rm N}_{1}$ records the only case arising in Theorem \ref{t:main_disc} with $H$ non-geometric. As explained in Remark \ref{r:ndisc}, a label ${\rm G}_{i,j}$ indicates that $H$ is a geometric subgroup in the collection $\C_i$ (see Table \ref{t:subs}). 

\item[(b)] In the second column we give the pair $(G,H)$, together with any additional conditions on $G$ and $H$, as well as the characteristic $p$. For each geometric configuration ${\rm G}_{i,j}$, the embedding of $H$ in $G$ can be read off from the subscript $i$, which gives the specific $\C_i$ family containing $H$. In the one non-geometric case, the subgroup $A_5.2$ is embedded in $C_{10}$ via the $20$-dimensional $KH^{\circ}$-module $\L^3(U)$, where $U$ is the natural module for $H^{\circ}$ (see Theorem \ref{t:ngeom}). In addition, in the second column we write $T_i$ to denote an $i$-dimensional torus, and the notation $(\dagger)$ has the same meaning as in Table $3$ (see Remark \ref{r:tab5}(d)).

\item[(c)] Set $J = [H^\circ,H^\circ]$, so either $J$ is semisimple, or $H^\circ$ is a maximal torus of $G$ and $J=1$ (the latter occurs in the cases ${\rm G}_{2,1}$ and ${\rm G}_{2,5}$). In the cases where $J$ is semisimple, the third column gives the highest weight of a $KJ$-composition factor $V_1$ of $V$ (in the two exceptional cases, we just write ``$-$"). Similarly, in the fourth column we give the highest weight of $V$. Note that the case appearing in the first row (labelled ${\rm G}_{1,1}$) is Ford's interesting example from \cite{Ford1}, and so it is convenient to refer the reader to Example \ref{ex:ford} for the various conditions on the highest weight which are needed for the irreducibility of $V|_{H}$.

\item[(d)] In the final column, $\kappa$ is the number of $KH^\circ$-composition factors of $V|_{H^\circ}$.
\end{itemize}
\end{re}

\vspace{45pt}

\begin{center}
Table 4. The irreducible triples with $G$ classical, $H$ disconnected \\
and maximal, and $V|_{H^\circ}$ reducible
\end{center}

\vspace{-2mm}

{\tiny
  \setlength{\tabcolsep}{4.95pt}
\begin{longtable}{llllc} 
no. & $H<G$ & $V_1$ & $V$ & $\kappa$ \\ 
& & & & \\
${\rm G}_{1,1}$ & $D_n.2<B_n, p \ne 2$ & \hspace{-2.2mm} 
  \begin{tikzpicture}[scale=.25, baseline=(current bounding box.base)]
   \tikzstyle{every node}=[font=\tiny]
   \draw (0,1.4) node[anchor=north] {$a_1$};
   \draw (2,1.4) node[anchor=north] {$a_2$};
   \draw (5.4,1.4) node[anchor=north] {$a_{n-2}$};
   \draw (10.8,1.2) node[anchor=east] {$a_{n-1}$};
   \draw (12.6,-1.2) node[anchor=east] {$a_{n-1}+1$};
    \foreach \x in {0,...,1}
    \draw[xshift=\x cm,fill=black] (\x cm,0) circle (.25cm);
    \draw[xshift=2 cm,fill=black] (4,0) circle (.25cm);
    \draw[xshift=6 cm,fill=black] (30: 17 mm) circle (.25cm);
    \draw[xshift=6 cm,fill=black] (-30: 17 mm) circle (.25cm);
    \draw[dotted,thick] (2.55cm,0) -- +(3cm,0);
    \foreach \y in {0.1,...,0.1}
    \draw[xshift=\y cm] (\y cm,0) -- +(1.6cm,0);
    \draw[xshift=6 cm] (30: 2mm) -- (30: 16mm);
    \draw[xshift=6 cm] (-30: 2mm) -- (-30: 16mm);   
    \end{tikzpicture}
    & \hspace{-2.2mm} 
    \begin{tikzpicture}[scale=.25, baseline=(current bounding box.base)]
  \tikzstyle{every node}=[font=\tiny]
  \draw (0,1.4) node[anchor=north] {$a_1$};
   \draw (2,1.4) node[anchor=north] {$a_2$};
   \draw (5.8,1.4) node[anchor=north] {$a_{n-1}$};
   \draw (8,1.4) node[anchor=north] {$a_n$};
    \foreach \x in {0,...,1}
    \draw[xshift=\x cm,fill=black] (\x cm,0) circle (.25cm);
    \draw[xshift=3 cm,fill=black] (3 cm, 0) circle (.25cm);
    \draw[xshift=5 cm,fill=black] (3 cm, 0) circle (.25cm);
    \draw[dotted,thick] (2.55cm,0) -- +(3cm,0);
    \foreach \y in {0.1,...,0.1}
    \draw[xshift=\y cm] (\y cm,0) -- +(1.6cm,0);
    \draw (6.2 cm, .1 cm) -- +(1.6 cm,0);
    \draw (6.2 cm, -.1 cm) -- +(1.6 cm,0);
    \draw (7.3,0) --+ (-0.7,0.4);
    \draw (7.3,0) --+ (-0.7,-0.4);
         \end{tikzpicture}  
& 2 \\
& & & See Example \ref{ex:ford} & \\
\end{longtable}}

\end{document}